\documentclass[5p]{elsarticle}
\biboptions{sort&compress}

\usepackage[draft,author=]{fixme}
\fxusetheme{colorsig}

\usepackage{amsmath,amssymb,amsthm}
\usepackage{mathtools}
\usepackage{hyperref} 

\usepackage{bookmark}
\usepackage{doi}
\usepackage{graphicx}
\DeclareGraphicsExtensions{.eps,.tif}
\graphicspath{{./revimg/}}

\usepackage{enumerate}
\usepackage{paralist}
\usepackage{subfig}
\usepackage{natbib}

\usepackage{hyphenat}
\newtheorem{thm}{Theorem} 
 
\newtheorem{prop}[thm]{Proposition} 

\newcommand{\field}[1]{ \ensuremath{\mathbb{#1}}}
\newcommand{\C}{ \ensuremath{\field{C}}}
\newcommand{\R}{ \ensuremath{\field{R}}}
\newcommand{\Q}{ \ensuremath{\field{Q}}}
\newcommand{\N}{ \ensuremath{\field{N}}}
\newcommand{\Z}{ \ensuremath{\field{Z}}}
\newcommand{\T}{ \ensuremath{\field{T}}}

\usepackage{lineno}

\renewcommand{\phi}{\varphi}
\renewcommand{\epsilon}{\varepsilon}
\newcommand{\norm}[1]{ \ensuremath{\left\lVert{#1}\right\rVert}}
\newcommand{\floor}[1]{\ensuremath{\left\lfloor{#1}\right\rfloor}}
\newcommand{\abs}[1]{ \ensuremath{\left\lvert{#1}\right\rvert}}
\newcommand{\remark}{{\textbf{Remark. }}}

\DeclareMathOperator{\Vol}{Vol}

\begin{document} 
\title{Geometry of the ergodic quotient reveals coherent structures in flows}

\author[ucsb]{Marko Budi\v{s}i\'c\corref{cor1}}
\ead{mbudisic@engr.ucsb.edu}
\author[ucsb]{Igor Mezi\'c}
\ead{mezic@engr.ucsb.edu}

\cortext[cor1]{Corresponding author}
\address[ucsb]{Department of Mechanical Engineering and Center for Control, Dynamical Systems and Computation, \\University of California, Santa Barbara, USA}

\date{\today}
\begin{abstract}
  Dynamical systems that exhibit diverse behaviors can rarely be completely understood using a single approach. However, by identifying coherent structures in their state spaces, i.e., regions of uniform and simpler behavior, we could hope to study each of the structures separately and then form the understanding of the system as a whole. The method we present in this paper uses trajectory averages of scalar functions on the state space to: \begin{inparaenum}[\upshape(\itshape a\upshape)]
\item identify invariant sets in the state space,
\item form coherent structures by aggregating invariant sets that are similar across multiple spatial scales.
\end{inparaenum}
First, we construct the ergodic quotient, the object obtained by mapping trajectories to the space of trajectory averages of a function basis on the state space. Second, we endow the ergodic quotient with a metric structure that successfully captures how similar the invariant sets are in the state space. Finally, we parametrize the ergodic quotient using intrinsic diffusion modes on it. By segmenting the ergodic quotient based on the diffusion modes, we extract coherent features in the state space of the dynamical system. The algorithm is validated by analyzing the Arnold-Beltrami-Childress flow, which was the test-bed for alternative approaches: the Ulam's approximation of the transfer operator and the computation of Lagrangian Coherent Structures. Furthermore, we explain how the method extends the  Poincar\'e map analysis for periodic flows. As a demonstration, we apply the method to a periodically-driven three-dimensional Hill's vortex flow, discovering unknown coherent structures in its state space. In the end, we discuss differences between the ergodic quotient and alternatives, propose a generalization to analysis of (quasi-)periodic structures, and lay out future research directions.
\end{abstract}

\maketitle

\section{Introduction}
\label{sec:intro}
Historically, a lot of work in dynamical systems focused on understanding a single type of behavior in a model: elliptical zones, hyperbolic trajectories, chaotic attractors, mixing regions, etc. As computers become more powerful, and experimental methods, such as Particle Image/Tracking Velocimetry, more precise, we are able to study large, complicated systems, whose state spaces comprise many different coexisting coherent structures: regions of uniform and simpler behavior, which might be well understood individually. To apply our understanding to the entire system, however, we first need to identify where coherent structures lie in the state space. This task can be made difficult by the quantity of coherent structures, fractal boundaries between them, or the high dimension of the state space.

Over the past decade, several different interpretations of coherent structures have appeared, each serving as a foundation for algorithms which extract them from the flow or from trajectory data. Coherent structures have been identified as:
\begin{inparaenum}[\upshape(\itshape i\upshape)]
\item distinguished trajectories that organize behavior of the flow nearby,
\item regions of state space between barriers of dynamical transport,
\item sets of initial conditions that are not dispersed by the flow,
\item sets of trajectories on which all flow\hyp{}invariant functions take constant values.
\end{inparaenum}
All the mentioned definitions identify sets that are dynamically invariant, in some sense, however, the choice of atomic objects, i.e., barriers of transport vs.\ initial conditions vs.\ trajectories, influences the type of approximation we obtain when using algorithms finite in time, space, and precision.

Identification of coherent structures based on barriers to transport includes both classical geometric studies of invariant manifolds and newer approaches that attempt to generalize invariant manifolds to aperiodic, transient flows. The study of Lagrangian Coherent Structures (LCS) \cite{Haller:2000us,Haller:2011kr} focuses on structures that take the roles of hyperbolic invariant manifolds in time-varying flows. As a consequence, they are dominantly studied through finite-time Lyapunov exponent fields whose local extrema are tracked to obtain proxies to LCS \cite{Haller:2002bf,Shadden:2005vn,Joseph:2002vi}. 

While LCS-based methods identify distinguished hyperbolic trajectories as organizing structures, a recent paper \cite{Mezic:2010kh}, coauthored by the second author, proposed to identify both hyperbolic and elliptic behavior in a finite-time, transient context. The approach, termed mesochronic velocity analysis, studies the Jacobian of the flow averaged over short segments of trajectories. 

Both LCS and mesochronic velocity fields identify distinguished trajectories that organize surrounding behavior. ``Surrounding'' is usually interpreted in terms of regions between neighboring distinguished trajectories. Such an interpretation is difficult to generalize to state spaces of higher dimension, where trajectories are not of codimension-one and cannot be used to segment the state space, without additional assumptions, e.g., symmetries of solutions.

Approaches that aggregate smaller flow features into larger coherent structures are the alternative to analyzing barriers that separate trajectories. A prominent aggregation method is the computation of invariant sets through eigenfunctions of the transfer (Perron-Frobenius) operator, a linear operator that captures evolutions of densities of points carried by dynamics \cite{Froyland:2003jj, Froyland:2009ti,Froyland:2010jo, Froyland:2011tt}. Numerically, the approach most commonly proceeds by the Ulam's method: approximation of the system by a stochastic Markov chain, acting on a discretized state space. The eigenvectors of the Markov chain transition matrix approximate the eigenfunctions  of the transfer operator. The Ulam's method requires a computation of short trajectories started from a large number of initial conditions within each discretization cell. Despite its prominence, Ulam's method can be difficult to apply to, e.g., high\hyp{}dimensional systems, and some experimental setups. A large number of state space dimensions may render computations infeasible, since the number of elements in Ulam's matrix grows exponentially with the dimension of the state space. On the other hand, if resetting and precise preparation of initial conditions is not possible, e.g., in certain fluids experiments, it can be more feasible to run fewer experiments for longer periods of time. 

The method that does use longtrajectories as a direct input has been pursued in \cite{Mezic:1999fu,Mezic:2004is}. It groups trajectories based on values of invariant functions along them. This approach interprets invariant functions as eigenfunctions of the Koopman operator, which captures how values of functions evolve along trajectories of the system. Koopman and the aforementioned Perron-Frobenious operator form a dual pair in the appropriate function spaces. Knowing eigenfunctions of the Koopman operator allows us to detect invariant sets in the state space, as level sets of Koopman eigenfunctions divide the state space into invariant partitions. It is possible to compute them without approximating the operator itself: starting from any observable, i.e., a scalar function on the state space, and averaging its values along trajectories projects the observable to the invariant eigenspace of the Koopman operator. Choosing observables from a function basis, averaging them, and forming joint level sets of averaged observables results in finer and finer stationary partitions of the state space. Ultimately, the process converges to the ergodic partition of the space: the unique, finest partition into invariant stationary sets. It was noted in \cite{Levnajic:2010gq} that even a small number of observables can be used to approximate the ergodic partition well. It remained unclear, however, how to select such a small set of observables that captures a lot of detail except by trial and error.

The algorithm proposed in this paper follows the analysis of dynamics using averaged observables. It does not aim to ``guess'' an observable whose time average reveals the most information about the coherent structure; rather, it \emph{constructs} a suitable set of invariant functions from averages of a function basis. To explain the method, we expand on the representation of dynamics using the ergodic quotient \cite{Levnajic:2010gq}: each trajectory is represented by a sequence of time\hyp{}averaged observables taken from a Fourier function basis. Collection of all such sequences, the ergodic quotient, is equivalent to the ergodic partition, but more practical to work with, as analytic tools on sequence spaces are very well studied. To analyze the ergodic quotient, we endow it with a distance function, adapted from \cite{Mathew:2011ev}, that identifies two trajectories as coherent if they spend similar fractions of their evolutions within any particular set in the state space. As a result, the problem of dynamical analysis of flows is converted into a geometric analysis of the ergodic quotient, with a metric structure given by the empirical distance. Our approach is similar in spirit to analysis of integrable systems through geometry of their integrals of motion \cite{Fomenko:2004}. 

To analyze geometry of the ergodic set we used the Diffusion Maps algorithm \cite{Coifman:2006cy}, which retrieves dominant modes of diffusion on objects with metric structure, e.g., branched manifolds. We introduced the use of Diffusion Maps as a tool in \cite{Budisic:2009iy}; here we explain its results in terms of the ergodic quotient. To validate the algorithm, we apply it to the Arnold-Beltrami-Childress (ABC) flow, which was previously used as a test-bed for algorithms based on the Ulam's method \cite{Froyland:2009ti} and on LCS \cite{Haller:2001vh}. The second example, the periodically-forced Hill's vortex flow, demonstrates the application of the algorithm to time-dependent flows.

When applied to analysis of time-periodic dynamical systems, our method extends the familiar analysis of the Poincar\'e map, including the averaged information about trajectories between two piercings of the Poincar\'e surface. Building on the application to time-dependent systems, we explain how the algorithm could be easily extended to the analysis of invariant features of periodic and quasi-periodic sets in the state space \cite{Mezic:2004is}. In this case, instead of ergodic averages along trajectories, we use harmonic averages along trajectories, which can also be interpreted through eigenspaces of the Koopman operator. While the ergodic averaging projected functions onto invariant eigenspace of the Koopman operator, the harmonic averaging projects functions to eigenspaces associated to other eigenvalues.

There are several benefits to the analysis based on averaging observables along trajectories. Since we use the entire trajectory evolution, we do not need to access the entire state space directly, e.g., we can place initial conditions on a lower\hyp{}dimensional surface and still obtain the full information about coherent structures intersected by that surface. Furthermore, compared to Ulam's method, we require fewer initial conditions for analysis. These features makes the method applicable to experimental setups where resetting and initialization of experiments are restricted or costly. Moreover, when there are low-dimensional surfaces of interest, such restricted initialization of trajectories makes it possible to analyze even high\hyp{}dimensional systems without the exponential growth in number of initial conditions. Finally, approximation of eigenfunctions of the Koopman operator does not require passing through a stochastic approximation of the deterministic dynamical system; instead projections of observables to eigenfunctions are evaluated directly by averaging.

Although both of our examples in this paper originated as models of fluid flows, the ergodic quotient analysis can be applied to a broad range of dynamical systems on compact state spaces. The incompressible fluid flows and hamiltonian systems, however, are two classes of systems that commonly have many invariant sets in their state spaces, arranged in intricate structure, therefore, those classes provide the most natural targets for our analysis. Nevertheless, dissipative systems can also be studied using the ergodic quotient, in which case the method identifies how basins of attraction are arranged in state spaces.

The paper is structured as follows: the Section \ref{sec:theory} presents the theory of the ergodic quotient, the empirical distance, and the diffusion coordinates, followed by the application to the autonomous ABC flow.  Section \ref{sec:time-varying} demonstrates how application to non-autonomous, periodically forced flows improves on the classical analysis of the Poincar\`e map and supports the case by a numerical analysis of the periodically-forced 3D Hill's vortex flow. In Section \ref{sec:discussion} we discuss how our method compares with alternatives, as well as present our suggestions for application and future research. The Appendices analyze numerical properties of the averaging algorithm and the empirical distance, and present details about the Diffusion Maps that are necessary for an implementation as a computer code.

\section{The Ergodic Quotient}
\label{sec:theory}

Geometry of the ergodic quotient enables comparison of trajectories of dynamical systems across multiple spatial scales. To be able to study it, we move through three different spaces: the state space, the space of averaged observables, where the ergodic quotient is constructed, and the space of diffusion coordinates, where the ergodic quotient is described by independent, intrinsic coordinates. Before delving deeper, we give an overview of the process, with italicized terms defined later in the text.

\begin{figure}[htpb]
  \centering
  \includegraphics[width=80mm]{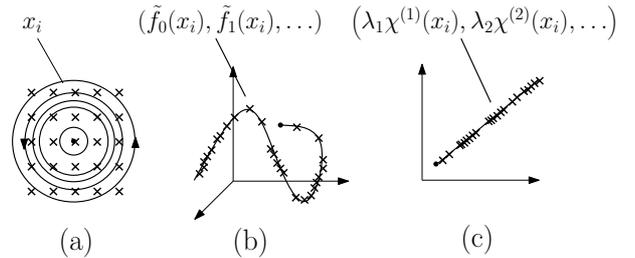}
  \caption{A sketch of (a) the state space portrait, (b) the ergodic quotient in the space of averaged observables, and (c) the ergodic quotient in the space of diffusion coordinates. Crosses depict initial conditions used for numerical computation. Note that crosses on the ergodic quotient represent entire trajectories started from associated initial conditions on the state space.}\label{fig:sampling}
\end{figure}

To measure how similar two trajectories are in the state space (Figure \ref{fig:sampling}.a), we use the \emph{empirical distance}, which compares trajectories based on their residence times in spherical sets in the state space. On the state space, however, the empirical distance is not practically computable, so we map trajectories to a space where it is. By selecting observables $f_k$ from a Fourier function basis and averaging them along trajectories, each trajectory gets mapped to a point on the \emph{ergodic quotient}, which comprises sequences of \emph{trajectory averages}\footnote{Terms \emph{time average} and \emph{trajectory average} are used interchangeably. In literature, \emph{ergodic average} and \emph{C\'esaro average} are also used.} $\tilde f_k$ for every trajectory in the state space (Figure \ref{fig:sampling}.b). We endow the sequence space with a metric induced by a $H^{-s}$ Sobolev space norm. The $H^{-s}$ metric is topologically equivalent to the empirical distance, yet it has a benefit that it is easily computable as a weighted euclidean metric. The ergodic quotient is typically low\hyp{}dimensional, therefore its representation in a (infinitely-dimensional) sequence space is not economical. To address this issue, we use a nonlinear change of coordinates to represent the quotient in \emph{diffusion coordinates}, preserving the intrinsic geometry of the quotient. Diffusion coordinates are an orthogonal, efficient, scale-ordered, low\hyp{}dimensional representation of the ergodic quotient that has an effect of straightening the ergodic quotient (Figure \ref{fig:sampling}.c). The euclidean distance in diffusion coordinates represents the  \emph{diffusion distance} along the ergodic quotient. Diffusion distance is an intrinsic distance, locally consistent with empirical distance; for the sake of introduction, it can be thought of as a robust version of the geodesic distance along the quotient.

Computationally, the construction starts with a finite set of trajectories $x_i$ on the state space. They are mapped to the ergodic quotient by averaging a finite subset of the Fourier basis along them, therefore, the points on the numerical ergodic quotient are vectors of trajectory averages. Based on trajectory averages, we compute the matrix of pairwise $H^{-s}$ distances between trajectories, which forms the input data to the Diffusion Maps algorithm. It computes the diffusion coordinates of each trajectory by solving a matrix eigenproblem. Starting trajectories are then grouped into coherent structures, using a $k$-means algorithm applied to the diffusion coordinates of trajectories. To visualize the coherent structures in the state portrait, we either color trajectories based on their $k$-cluster membership, or we color them based on a single diffusion coordinate.

\subsection{Construction: From state space portrait to the ergodic quotient}
\label{sec:empmeas}

Take a compact euclidean manifold $\mathcal M \in \R^D$ as the state space of dynamics. The flow map is a composition (semi)-group of continuous functions $\phi_t(x) : \mathcal M \to \mathcal M$, with time $t \in \mathcal T$. For ODE-generated flows, the time-like parameter is a real number $\mathcal T = \R^+_0$, and the semigroup is assumed to be continuous in $t$. In map-generated flows, time progression over $\mathcal T = \Z^+_0$ is interpreted as the iterated application of the map $\phi(x)$, i.e., $\phi_t(x) = \phi[ \phi_{t-1}(x)]$, with $\phi_0(x) = x$. The space of observables $\mathcal F$ is a set of Borel-measurable functions $f : \mathcal M \to \C$; in this paper we will take $\mathcal F$ as the set of continuous functions on $\mathcal M$.

The geometric analysis of dynamics studies the state space portrait of the system, i.e., the collection of all trajectories in the state space, $\left\{ \left(\phi_t(x)\right)_{t\in \R} : \forall x \in \mathcal M \right\}$. Chaotic behavior, however, implies that individual trajectories are non-robust to errors in initial conditions $x$. Despite the lack of robustness of trajectories, the averages of observables along them are robust, even in chaotic regimes \cite{Sigurgeirsson:2001wi,ChaosBook}. As our approach is based on averaging, we expect the computations to remain robust even in chaotic regimes, which is the main motivator for this approach. Presented introductory topics follow the exposition given in \cite[][\S 4]{Katok:1995ul}.

A finite-time average of an observable $f$ is given by
\begin{equation*}
\frac{1}{T}\int_0^T \left[f \circ \phi_t \right](x) dt.
\end{equation*}
Let $\Sigma \subset \mathcal M$ be the set of points on which the limit as $T\to\infty$ exists.
On $\Sigma$, we can define the infinite time average
\begin{equation*}
\tilde f(x) \triangleq \lim_{T \to \infty}\frac{1}{T}\int_0^T \left[f \circ \phi_t \right](x) dt.
\end{equation*}
If the flow map preserves a measure $\lambda$, the set $\mathcal M \backslash \Sigma$ will be negligible in measure, $\lambda(\mathcal M \backslash \Sigma) = 0$, by Birkhoff's Ergodic Theorem. As hamiltonian systems and divergence-free flows preserve the Lebesgue measure, it is unlikely that an initial condition for which time averages do not converge will be selected when analyzing such systems. The same argument can be made for systems that preserve a Lebesgue-continuous measure, which contains an even larger class of systems. Consequently, we will equate $\Sigma$ and $\mathcal M$ for purposes of this paper.

The infinite-time average $f \mapsto \tilde f(x)$ is a positive, linear, continuous functional, and the Riesz representation theorem asserts that the averaging functional can be represented by a Borel probability measure $\mu_{x}$, the \emph{empirical measure}, attached to the initial condition $x$:
\begin{align}\label{eq:spacetimeavg}
  \int_\mathcal{M} f d\mu_{x} = \tilde f(x).
\end{align}
The empirical measures of sets $\mu_x(E)$  can be interpreted as the fraction of evolution that the trajectory evolving from $x$ spends inside the set $E$ (residence or mean-sojourn time). For regular orbits, empirical measures are masses supported on fixed points or spread out along periodic orbits, while on chaotic sets they may have a positive-area support. It is worth noting that empirical measures $\mu_x$ are \emph{instantaneous} invariants of the system, despite our use of an asymptotic process in their construction. They play a similar role in the measure-theoretic description of dynamics as the trajectories do in the geometric description. 

Instead of working with potentially singular empirical measures directly, we will work with their weak representatives. Fix a countable basis of continuous functions $f_k$ on $\mathcal M$, indexed by a multi\hyp{}index $k$, e.g., Fourier harmonics with $D$-dimensional wavevectors $k \in \Z^D$. Let the \emph{ergodic quotient map} be $\pi : \mathcal M \to l^\infty(\Z^D)$, given by
\begin{align*}
  \pi(x) := ( \dots, \tilde f_k(x), \dots )_k, \quad k \in \Z^D.
\end{align*}
Action of $\pi$ at $x$ can be interpreted as forming a weak representation of the empirical measure $\mu_x$. The set of all such sequences, generated by averaging from each initial condition, is the \emph{ergodic quotient} $\xi := \pi(\mathcal M)$. It will play the same role that the state-space portrait plays in geometric analysis, with the added benefit of robustness in all dynamical regimes.

Strictly, every point on the ergodic quotient should be thought of as the set of trajectory averages of \emph{all} continuous observables: by selecting a basis we merely pick a representation of the ergodic quotient in the particular sequence space $l^\infty(\Z^D)$. It was shown in \cite{Mezic:2004is} that all such representations are equivalent; therefore, we will speak of any of them as \emph{the} ergodic quotient when there is no ambiguity.

Conceptually, we have moved from points in the state space $x$ to associated averages of observables in two steps:
\begin{align}
  \label{eq:EQ}
  x \xrightarrow{\text{I}} \mu_x \xrightarrow{\text{II}} ( \dots, \tilde f_k(x), \dots )_{k\in\Z^D}.
\end{align}
The representation steps (I) and (II) in \eqref{eq:EQ} are not bijective in general. 

Step (I), at the very least, associates all points on the same trajectory with the same empirical measure; however, the pre\hyp{}images of empirical measures might contain more than a single trajectory. In \cite{Budisic:2009iy} we discussed the approximation of the ergodic partition, which is the (unique) partition of the state space into ergodic sets. An invariant subset of the state space $S \subset \mathcal M$ is called ergodic when the flow $\phi_t$ restricted to it, $\phi_t:S \mapsto S$, is an ergodic system. Since trajectory averages of ergodic systems do not depend on initial conditions, it follows that all points $x \in S$ have the same empirical measure. In this sense, the pre\hyp{}image of the ergodic quotient is the ergodic partition of the state space. Such an outcome is  desirable, as ergodic sets are minimal robust atomic objects that contain (non-robust) trajectories. 

Step (II) is bijective only if the entire function basis can be used to construct the ergodic quotient. On continuous state spaces which have infinite bases, we will need to truncate the function set to implement the algorithm on a computer: the magnitude of error introduced by basis truncation will depend on the distance structure on the ergodic quotient, discussed in the Section \ref{sec:distance}.

\remark Unless we are looking for the ergodic quotient of the entire system, we can place initial conditions only in a region of space that is of a particular interest. Moreover, only one initial condition from an ergodic set is required to represent the empirical measure for the entire ergodic set, as any trajectory in the ergodic set will correctly sample the associated empirical measure. Therefore, it is sufficient to sample the initial conditions from the surface that intersects coherent features we wish to explore. After selecting the initial conditions and the set of observables to average, we compute trajectories starting with each initial condition until time averages along them converge (see \ref{sec:convergence}), thus obtaining the approximate mapping of the state space into the ergodic quotient.

\subsection{Metric structure: Empirical distance}
\label{sec:distance}

Trajectories are aggregated into coherent structures by a criterion of similarity embodied in a distance function between trajectories. The choice of the distance function determines the type of coherent structures we obtain. In this work, we chose to work with the empirical distance, which compares two trajectories based on residence times \cite{Mathew:2011ev}. Instead of working with the empirical distance directly on the state space, we use a distance on the ergodic quotient that is equivalent to it, thus converting the analysis of dynamics into analysis of geometry of the ergodic quotient.

The empirical distance was originally used to quantify how closely a trajectory samples an a priori distribution on the state space, i.e., to compare an empirical measure with a fixed prior measure on the state space. The same distance can be used to compare two empirical measures and estimate how similar the processes that generate measures are. Define the empirical distance $\mathcal D$ between measures $\mu_x$, $\mu_y$ by 
\begin{align*}
  \mathcal{D}(\mu_x,\mu_y)^2 = \int_0^1 \int_{\mathcal{M}} \left\{\mu_x\left[B(p,r)\right] - \mu_y\left[B(p,r)\right]\right\}^2 dp dr,
\end{align*}
where $B(p,r) \subset \mathcal{M}$ are spherical sets of center $p$ and radius $r$ normalized such that $B(p,1) \supset \mathcal{M}$ and $\Vol \mathcal M = 1$. Note that integration is over all spherical sets, without normalization by the sets' volumes, resulting in higher sensitivity to differences across larger spatial scales.

Integration over uncountably many spherical sets is practically infeasible, however, the metric induced by the norm on a negative-order Sobolev space of measures can be used instead. The distance
\begin{align}
  \begin{aligned}
    \mathcal{D}_{-s}(\mu_x,\mu_y) &\triangleq \norm{\tilde F(x) - \tilde F(y)}_{2,-s}\\
    \norm{\tilde F(x) - \tilde F(y)}_{2,-s}^2 &= \sum_{k \in \Z^D}
    \frac{ \abs{\tilde f_k(x) - \tilde f_k(y)}^2 }{ \left[ 1 +
        \left(2\pi \norm{k}_2 \right)^2\right]^s },
  \end{aligned}\label{eq:negsobdist}
\end{align}
is easily computed as a weighted sum of Fourier coefficients, $\tilde F(x) \triangleq \left( \dots, \tilde f_k(x), \dots \right)_{k \in \Z^D}$ of measures involved. The order $s$ of the Sobolev space is selected based on the dimension $D$ of the domain of measures $\mathcal M$  as $s = (D+1)/2$. Using $\mathcal D_{-s}$ instead of $\mathcal D$ does not change the resulting topology, as the empirical metric and the $H^{-s}$ with the order $s = (D+1)/2$ are equivalent, i.e., there exist positive constants $c$ and $C$ such that for all $\mu_x$ and $\mu_y$ it holds that
\[
c \mathcal D(\mu_x,\mu_y) \leq \mathcal D_{-s}(\mu_x,\mu_y) \leq C \mathcal D(\mu_x,\mu_y).
\]

When $\mu_x$ and $\mu_y$ are empirical measures, the Fourier coefficients can be computed as trajectory averages of observables, due to  \eqref{eq:spacetimeavg}, which is the core observation on which we base our method. Therefore, we choose observables  $f_k$ as Fourier basis functions, 
\begin{equation}
  \label{eq:fourierfunction}
  f_k(x) = (2\pi)^{-D/2} \exp\left[ i 2\pi \sum_{d=1}^D k_d x_d\right],
\end{equation}
with $k_d$, and $x_d$ being components of the wave- and state-vectors, respectively, assuming $\mathcal M = \T^D$ for simplicity \footnote{For arbitrary domains, observables are solutions of the Helmholtz problem on the domain. G.\ Mathew and I.\ Mezi\'c will address the general formula for the $H^{-s}$ distance in their upcoming work.}. Interpreted in the ergodic quotient framework, sequences $\tilde f_k(x)$ and $\tilde f_k(y)$ are points on the ergodic quotient corresponding to empirical measures $\mu_x$ and $\mu_y$.
 
For purposes of our method, it is important to capture the topology and the geometry of the ergodic quotient, not the values of involved quantities. Therefore, we will use the empirical distance $\mathcal D$ and the $H^{-s}$ Sobolev space distance $\mathcal D_{-s}$ interchangeably: the former in explaining conceptual implications of the metric structure, the latter in computational considerations.

In practice, we cannot compute averages of the entire function basis. The direct effect of the basis truncation is difficult to quantify; intuitively, the higher the cutoff $K$ is set, restricting wavevectors to $k \in [-K,K]^{D}$, the finer are the state space features we are able to resolve through empirical distance. In \ref{sec:finitewavevector} we show that truncation introduces at most $\mathcal O(1/\sqrt{K})$ error in the distance for any state space dimension $D$.

\remark From a signal processing perspective, the $H^{-s}$ distances correspond to a low\hyp{}pass filtering of empirical distributions prior to their comparison. Since numerical errors are more likely to manifest in averaging of higher-wavenumber modes, we expect that the use of such distances will improve the numerical stability of the analysis. For a brief analysis of these effects, see \ref{sec:convergence}.

\subsection{Geometry of the ergodic quotient}
\label{sec:geometry}

Equipping the ergodic quotient with a distance function provides a setting for studying the geometry of the space of invariant functions. The study of dynamics through ergodic quotient shows similarities with the study of dynamics of integrable systems through geometry of integrals of motion. The main difference is that the ergodic quotient is a constructive technique: we do not need to know the expressions for integrals of motion to construct the ergodic quotient or approximate it numerically. Moreover, construction of the ergodic quotient is possible even when the system is not integrable.

\begin{figure}[htpb]
  \centering
  \includegraphics[width=80mm]{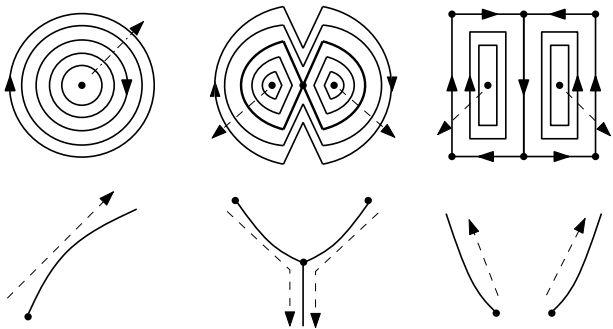}
  \caption{Geometry of the ergodic quotient under empirical distance. Sketches of matching features in state space portraits (top) and ergodic quotients (bottom) are given from left to right for a harmonic oscillator, a Duffing-type oscillator and a double gyre system. Dots indicate equilibria and dashed arrows indicate how moving an initial condition in the state space moves its image on the ergodic quotient.}\label{fig:eqsketch}
\end{figure}

In Figure \ref{fig:eqsketch} we demonstrate what the ergodic quotient might look like locally for several simple state-space portraits. In elliptic regions, invariant functions are a one-parameter family, mapping to a filament in the ergodic quotient. In Duffing-type dynamics, one-parameter invariants exist along libration trajectories in elliptic zones and along the revolution trajectories outside of the separatrix. Moving initial conditions close to separatrix from either side results in trajectories slowing down more and more as they pass the saddle, and consequently time averages of an observable will converge to the value that the observable attains at the saddle. As a result, filaments for librations and revolution trajectories meet in the ergodic quotient. In double-gyre flows, the filaments do not meet, since two adjacent elliptic cells share only two out of four saddles; as trajectories slow down next to all four, the time averages will converge to a weighted average of the function values at the saddles, which will be different for every adjoining elliptic cell.

\remark Fomenko in \cite{Fomenko:2004} used a similar construction, the Reeb graphs, to analyze integrable systems. Reeb graphs are graphical representations of connected components in level sets of Morse functions, in particular, the integrals of motion. For an integrable system, the joint level sets of integrals of motion are ergodic sets. Fomenko did not extend his work to systems for which we do not have expressions for integrals of motion. The ergodic quotient, therefore, can be interpreted as a setting for generalization of Fomenko's approach to non-integrable systems.

\subsection{Diffusion modes as global coordinates}
\label{sec:globalcoordinates}

Although the ergodic quotient is a subset of an infinite-dimensional sequence space, its main features could be approximated well in a lower number of dimensions, as illustrated in Figure \ref{fig:eqsketch}. To obtain a low\hyp{}dimensional representation, we will convert from the ambient coordinates, where axes are given by averaged observables, to intrinsic coordinates, where axes are independent parameters that parametrize the ergodic quotient. As intrinsic coordinates, we chose the modes of diffusion, or heat spread, along the ergodic quotient, which provide an efficient,  orthogonal low\hyp{}dimensional representation. Numerically, we pass from ambient coordinates to diffusion modes using the Diffusion Maps algorithm \cite{Coifman:2006cy}.

Problems of approximating high\hyp{}dimensional data by low\hyp{}dimensional parametrizations are common in data mining and machine learning communities, where a low\hyp{}dimensional process, e.g., an object moving in the field of vision, is measured by high\hyp{}dimensional observation, e.g., image consisting of pixels. The usefulness of diffusion eigenfunctions, as a prominent example of laplacian-based methods, is well documented in such settings, e.g., \cite{Luxburg:2007bb}.

On a differentiable manifold, the diffusion operator is given by $A = e^{-t\Delta}$, where $\Delta$ is the Laplace-Beltrami operator and $t$ the time over which diffusion evolves. Spectrum of the diffusion operator is contained within the unit interval, with first eigenvalue always $\lambda_0 = 1$, and the rest of the eigenvalues ordered in decreasing order. Eigenfunctions of the diffusion operator $\chi^{(k)}$ are orthogonal and normalized to $\norm{\chi^{(k)}}_\infty$ amplitude with the trivial eigenfunction $\chi^{(0)} \equiv 1$. The numerical algorithm, The Diffusion Maps, retrieves the values of eigenfunctions sampled at finite number of points on the underlying manifold.

When we apply Diffusion Maps to sequences of time averages, the ergodic quotient plays the role of the differentiable manifold. The interpretation of numerical results as discretizations of the Laplace-Beltrami eigenfunctions is strictly valid only in the case when ergodic quotient is a differentiable manifold, as it seems to be for some simple systems, but not necessarily in the general case. Despite this fact, this paper demonstrates that Diffusion Maps are practically useful in analysis of dynamical systems and we expect that a consistent interpretation will be found even for sets without a differential structure.

The properties of diffusion eigenfunctions are demonstrated on simple examples in papers \cite{Coifman:2005bk,Lee:2010cz}, and detailed analysis of the mathematical properties can be found in \cite{Jones:2008cy,Jakobson:2001cf,Uhlenbeck:1976wb}. In short, we can expect the lowest-index eigenfunctions to be combinations of indicators on disconnected components of the ergodic quotient. The eigenfunctions of higher index behave similarly to Fourier modes on each of the components: first varying over coarser features and, as the order is further increased, varying over finer features.

In the spirit of representation sequence \eqref{eq:EQ}, we will add embedding into the eigenfunctions of the diffusion operator as the final step:
\begin{align*}
  x &\xrightarrow{\text{I}} \mu_x \xrightarrow{\text{II}} ( \dots, \tilde f_k(x), \dots )_{k\in\Z^D} \notag\\
  &\xrightarrow{\text{III}} (\lambda_1 \chi^{(1)}(x), \lambda_2 \chi^{(2)}(x), \dots),
\end{align*}
and refer to embedding $x \mapsto (\lambda_1 \chi^{(1)}(x), \lambda_2 \chi^{(2)}(x), \dots)$ as the \emph{diffusion coordinates} of the dynamics. The sketch of the embeddings was given at the beginning as Figure \ref{fig:sampling}.

\begin{figure}[htpb]
  \centering
  \includegraphics[width=40mm]{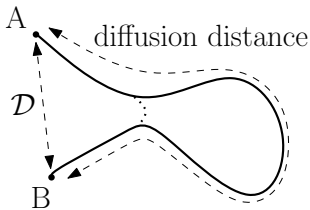}
  \caption{A sketch of distances in the space of trajectory averages. The full line represents the ergodic quotient, the chordal distance is the empirical metric, and the dotted line indicates the point where a small amount of noise can change the geodesic distance discontinuously, but not the diffusion distance.}\label{fig:distances}
\end{figure}

The euclidean distance in the diffusion coordinate representation carries a geometric meaning, which is different from, but related to the meaning of the distance in the ambient space of the ergodic quotient, spanned by averaged observables. The empirical distance $\mathcal D$ is the chordal distance in the ambient space, i.e., the length of a shortest path between points. Diffusion distance, like geodesic distance, takes into account only paths that lie within the ergodic quotient. However, while the geodesic distance is the length of the shortest path connecting two points, diffusion distance is, in a sense, the average length of all paths that connect the two points. This average is computed as the difference in heat flow reaching two points when heat is injected at any other point on the quotient, averaged over all possible positions of the heat source. A more detailed description can be found in \cite{Coifman:2006cy}.

For our purposes it is sufficient to state that averaging over all paths provides robustness to noise. If the noise introduces a short-cut between two parts of the manifold, as depicted in Figure \ref{fig:distances}, the geodesic distance changes discontinuously, however, the diffusion distance does not, as averaging over all paths ensures that a small amount of noise does not change the geometry significantly \cite[][{\S}4.2]{Lee:2010cz}.

The Diffusion Maps algorithm approximates the operator $e^{-t\Delta}$ by a $N \times N$ matrix, where $N$ is the number of points sampled from the ergodic quotient: in our case equal to the number of initial conditions. Discrete diffusion modes are obtained as eigenvectors of that matrix. The \ref{sec:diffusionmaps} provides more details about the implementation: the input to Diffusion Maps is a matrix of pairwise $\mathcal D_{-s}$ distances between computed trajectories, while the output is the set of diffusion coordinate vectors, assigning diffusion coordinates to each trajectory.

\subsection{Example: Steady ABC Flow}
\label{sec:abcflow}

We will demonstrate the use of the ergodic quotient on the Arnold-Beltrami-Childress (ABC) flow \cite{Dombre:1986td}. This flow was previously used to demonstrate identification of Lagrangian Coherent Structures in \cite{Haller:2001vh} and almost-invariant sets in \cite{Froyland:2009ti}. It is a steady, volume-preserving flow, evolving on a 3-torus. Despite its very simple system of ODEs, the state space of the ABC flow contains both regular vortices and chaotic zones. 

\begin{figure*}[htpb]
  \centering
  \subfloat[Six subsets in the state space corresponding to six clusters of aggregation. Subsets contain primary vortices of the ABC flow. The chaotic sea between vortices is the seventh cluster and appears as the void between vortices.]{
  \includegraphics[height=60mm]{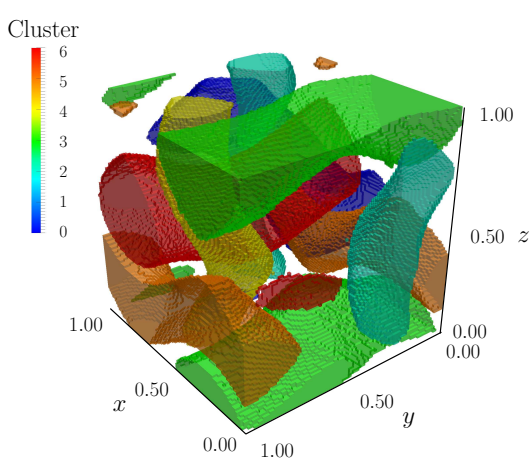}
\label{fig:abcstatespace6}}\hspace{.1\linewidth}
  \subfloat[Projection of ergodic quotient onto three out of ten diffusion coordinates used for clustering, axis labels are indices $k$ of coordinates $\lambda_k \chi^{(k)}$.  Points were colored according to membership in clusters. Cluster at the top of the figure corresponds to the chaotic sea.]{
  \includegraphics[height=60mm]{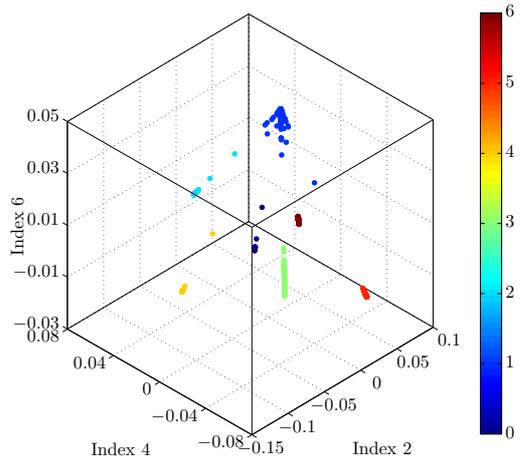}
\label{fig:abcergquot}}
  \caption{Six primary vortices extracted by k-means clustering ($k=7$) of projection of ergodic quotient onto first 10 diffusion coordinate. ABC flow \eqref{eq:abcmodel} was simulated with $A=\sqrt{3}$, $B=\sqrt{2}$, $C=1$, from $N=1000$ initial conditions uniformly distributed in basic periodicity cell $[0,1)^3$, with observables cut off at wavenumber $K=10$, and convergence tolerance $\text{ATOL}=2 \times 10^{-4}$.}\label{fig:sixvortices}
\end{figure*}

Trajectories of the ABC flow obey the following equations on $(x,y,z) \in \mathbb{T}^3$
\begin{align}    \label{eq:abcmodel}
  \begin{bmatrix}
    \dot x \\ \dot y \\ \dot z
  \end{bmatrix}
    &= A(x,y,z) \triangleq \frac{1}{2\pi}
    \begin{bmatrix}
      {\displaystyle A \sin 2\pi z + C \cos 2\pi y }\\
      {\displaystyle B \sin 2\pi x + A \cos 2\pi z }\\
      {\displaystyle C \sin 2\pi y + B \cos 2\pi x }
    \end{bmatrix}.
\end{align}
As in \cite{Haller:2001vh,Froyland:2009ti}, we used $A=\sqrt{3}$, $B=\sqrt{2}$, $C=1$, but the domain was rescaled to $\Vol \T^3 = 1$. Observables are the complex Fourier harmonics \eqref{eq:fourierfunction} with wavevectors $k \in [-10,10]^3 \cap \Z^3$. The trajectories were integrated using an adaptive solver (Radau) for at least $T=500$ per initial condition, but possibly longer to achieve tolerance $\text{ATOL}=2 \times 10^{-4}$ in convergence of averages for each observable and trajectory (see \ref{sec:convergence}).\footnote{The precise value of the tolerance chosen did not make a big difference in data obtained by the method for the ABC flow. Reducing tolerance further extends the computation on a relatively small number of trajectories (see Figure \ref{fig:convergence}) while most of the trajectories require very similar simulation times even if the tolerance is reduced by an order of magnitude.}

In short, the steps used for identification of coherent structures are:
\begin{enumerate}
\item Select initial conditions $x_i$,
\item Compute trajectories and time averages until convergence,
\item Compute $H^{-s}$ distance matrix \[(d_{ij}) = \mathcal D_{-s}( \mu_{x_i},\mu_{x_j} ),\]
\item Feed $(d_{ij})$ into Diffusion Maps to compute diffusion coordinates $\left( \lambda_1 \chi_1(x_i), \lambda_2 \chi_2(x_i), \dots\right)$,
\item Compute $k$-means clusters (optional, for visualization) .
\end{enumerate}

Connected components in the ergodic quotient correspond to sharply distinct coherent features. Information about all distinct components will rarely be visible from a single, or a few, diffusion coordinates. To quickly recover the dominant features in the flow, we can apply a clustering algorithm. Clustering based on the euclidean distance, e.g., the $k$-means algorithm, is justified since the euclidean distance approximates the diffusion distance on the ergodic quotient. 

The Figure \ref{fig:sixvortices} demonstrates the application of k-means clustering, with $k=7$, to the first ten diffusion coordinates on the ergodic quotient of the ABC flow. The six clusters correspond to six primary vortices that were identified in \cite[][Fig. 4]{Dombre:1986td}. Furthermore, these are the same structures visible in LCS analysis of the flow \cite[][Fig. 1]{Haller:2001vh} and the eigenvectors of the Ulam discretization of Perron-Frobenius operator \cite[][Fig. 6]{Froyland:2009ti}. This corroborates our assertion that the components of the ergodic quotient correspond to features typically identified as coherent structures in fluid flows.

In contrast to LCS and Ulam's methods, the initial conditions do not have to be placed in the entire state space, as mentioned in Section \ref{sec:empmeas}: it is sufficient to place them on a set that intersects interesting coherent structures that we want to analyze or visualize. In Figure \ref{fig:xvortex} we show a visualization based on $N=500$ initial conditions sourced from a rectangle $(y,z) \in [0.35,0.8] \times [0.6,0.9]$, on the $x=0$ face of the unit cell. This region intersects one of the six primary vortices shown in Figure \ref{fig:abcstatespace6}. Since we focused the initial conditions in a smaller region of state space, the same number of diffusion coordinates discerns smaller features in the state space, in this case the secondary $1:2$ vortices that wrap around the central vortex.

\begin{figure*}[htpb]
  \centering
  \subfloat[Five subsets identified by the k-means algorithm with $k=5$ acting on first ten diffusion coordinates.]{
\includegraphics[height=45mm]{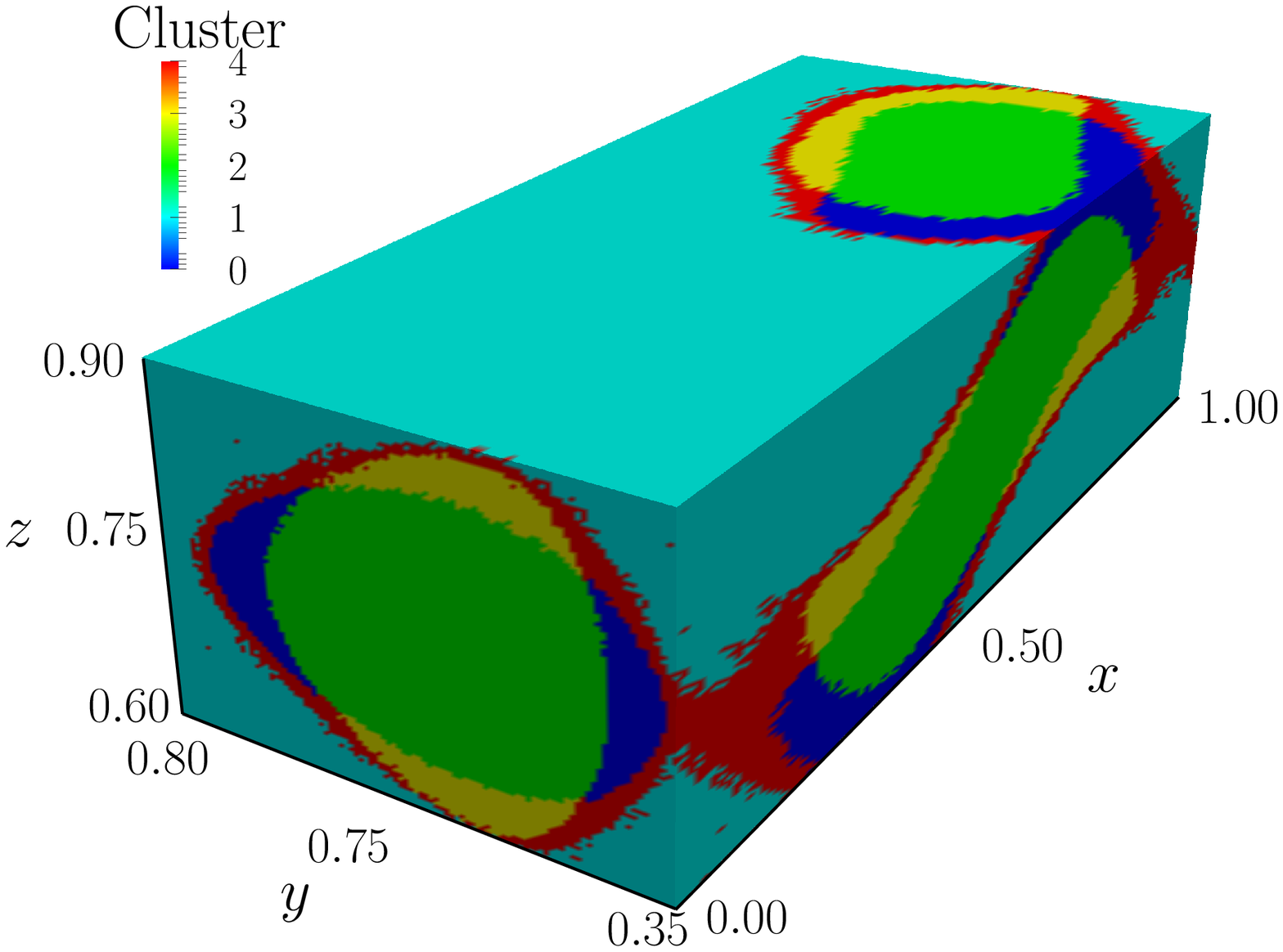}
\label{fig:xvortexstatespace5}
}\hfill
  \subfloat[Projection of the ergodic quotient onto three diffusion coordinates, axis labels are indices $k$ of coordinates $\lambda_k \chi^{(k)}$. Points were colored according to membership in clusters.]{
\includegraphics[height=45mm]{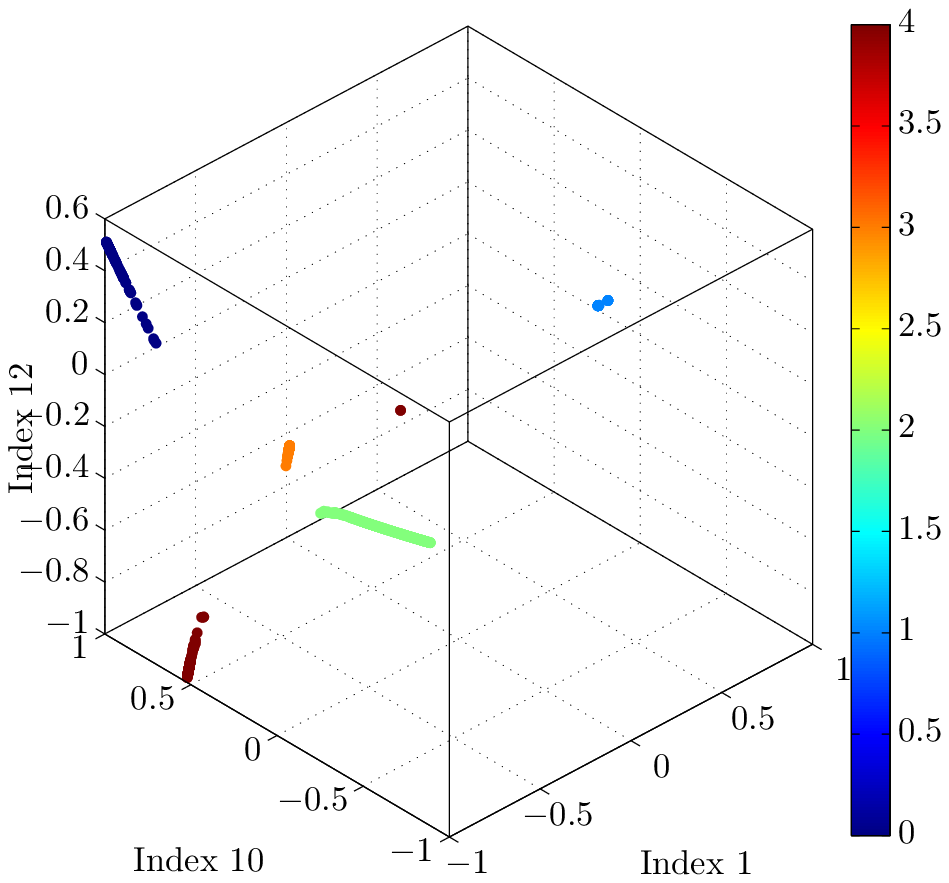}
\label{fig:xvortexergquot}
}\hfill
  \subfloat[Level sets of the 10th diffusion eigenvector. The chaotic sea has been cropped for clarity of display by thresholding the first diffusion eigenvector.]{
\includegraphics[height=45mm]{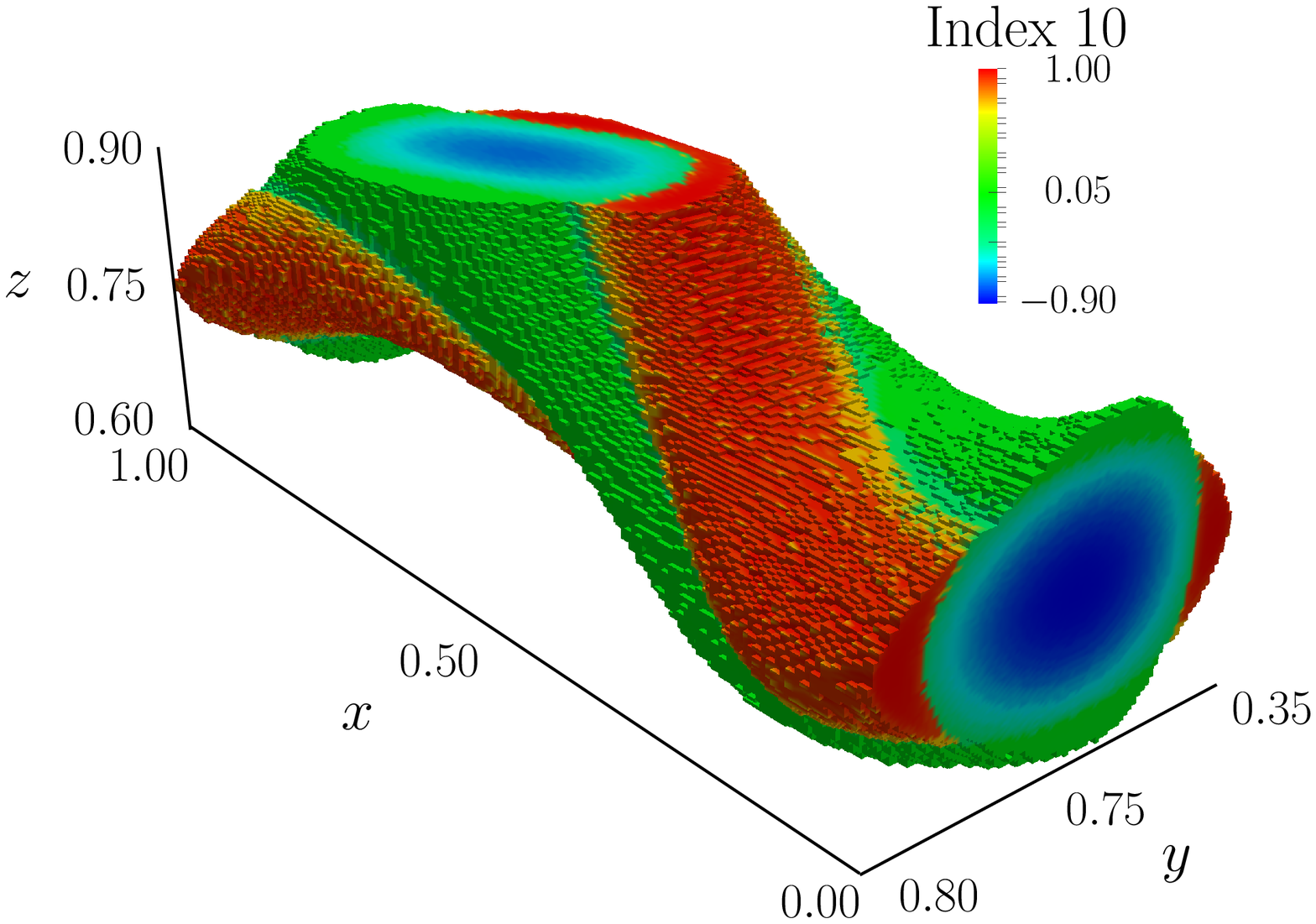}
\label{fig:xvortex_ec10}
}\caption{A primary vortex aligned with the $x$-axis of the ABC flow ($A=\sqrt{3}$, $B=\sqrt{2}$, $C=1$). Lobe-shaped structures that wrap around the central vortex correspond to secondary $1:2$ vortices. Results are based on $N=500$ trajectories initialized uniformly in rectangle $x=0$, $(y,z) \in [0.35,0.8] \times [0.6,0.9]$ , with observables cut off at wavenumber $K=10$, and convergence tolerance $\text{ATOL}=2\times10^{-4}$.}\label{fig:xvortex}
\end{figure*}

The Figure \ref{fig:xvortex_ec10} shows that some diffusion coordinates are local-ordering functions on the elliptic zones, acting as computationally constructed local integrals of motion. The 10th ergodic coordinate acts as an energy-like integral of motion, locally on the inner primary vortex. The blue filament in upper-left corner of Figure \ref{fig:xvortexergquot} indicates that the positive values of the 12th coordinate vary approximately continuously over the secondary vortex lobes, colored in blue in Figure \ref{fig:xvortexstatespace5}. This is an indication that a plot similar to Figure \ref{fig:xvortex_ec10} could be made to demonstrate local-ordering on a part of the secondary vortex, using the 12th coordinate.

\section{Extensions including time-varying features}
\label{sec:time-varying}

\subsection{Periodically-driven flows}
\label{sec:periodically-driven}

Although we have presented the construction of the ergodic quotient for the autonomous flows, in practice, nothing prevents one from applying the exact same computational procedure to a non-autonomous flow. In this Section, we explain that when system is non-autonomous, yet periodically-driven, one still obtains useful and non-trivial information about the structure of invariant sets.

The non-autonomous systems can be embedded in an autonomous system by extending the state space. Let $\mathcal M_e = \mathcal M \times \mathcal T$ be an extension of the state space $\mathcal M$. For periodically driven flows, $\mathcal T$ can be taken as $\T$, while for aperiodically driven flows and finite-time flows $\mathcal T$ can be taken as $\R$ and $[t_i,t_f] \subset \R$, respectively.

For a non-autonomous dynamical system defined by an ODE $\dot x = A(x, t)$, where $x \in \mathcal M$, we can define the extended, autonomous system through
\begin{align*}
  \dot x &= A(x,\tau), \\
  \dot \tau &= c,
\end{align*}
where $(x,\tau) \in \mathcal M_e$ and $c$ is a time-rescaling constant. For periodically-driven systems on compact $\mathcal M$, i.e., for $\mathcal T = \T$, the extended state space $\mathcal M_e$ is also compact. If, additionally, the flow map $\phi_t : \mathcal M_e \to \mathcal M_e$ of the system is continuous, the system $(\mathcal M_e, \phi_t)$ fulfills all the conditions that we required for the construction of the ergodic quotient map $\pi_e : \mathcal M_e \to l^\infty$, given in Section \ref{sec:empmeas}. 

The difference between an autonomous system on $\mathcal M$ and a periodically-driven system cast into an autonomous system on $\mathcal M_e$ is in the eye of observer. Typically, we will treat only the first factor of $\mathcal M_e$ as ``physical'' states, always keeping in mind that $\mathcal T$ was a formal construct accounting for the ``physical'' time. This is reflected in observables usually still available as functions $f : \mathcal M \to \C$. 

If the observables on $\mathcal M$ were chosen as Fourier functions $f_k : \mathcal M \to \C$, with $k \in \Z^D$, we can interpret them as a subset of Fourier observables $g_w : \mathcal M_e \to \C$, $w \in \Z^{d+1}$, 
\[
g_w(x,\tau) = C f_k(x) e^{i2 \pi w_{d+1} \tau},
\]
with the constant $C$ normalizing the observable to $\norm{g_w}_\infty = 1$.
When wave-vector $w$ is constrained to $w = (k,0)$, the identity $g_w(x,\tau) \equiv f_k(x)$ holds for any $\tau$. Consequently, the trajectory averages $\tilde g_w(x,\tau) = \tilde f_k(x)$ are then insensitive to the initial ``physical'' time $\tau$.

If we proceed with such a partial set of observables $g_w$, where $w = (k,0)$, $k \in \Z^D$, we cannot in general hope to obtain the full ergodic quotient by mapping $x \to \tilde g_w(x,\tau)$. Denote the partial quotient map by $\pi\vert_{\mathcal M} : \mathcal M_e \to l^\infty(\Z^D)$, defined through 
\[
\pi\vert_{\mathcal M}( x,\tau ) := (\dots, \tilde g_w(x,\tau),\dots)_w,
\]
for $w = (k,0)$ and $k \in \Z^D$, with the associated partial quotient $\xi\vert_{\mathcal M} \triangleq \pi\vert_{\mathcal M}(\mathcal M)$. Looking at pre-images of any trajectory-averaged observable, it follows that $\tilde g_w^{-1}(c) = \tilde f_k^{-1}(c) \times \T$. Therefore, for any point $p \in \Z^D$, with pre-image $\pi^{-1}(p)$ formed using observables $f_k$, the pre-image under partial quotient map is $\pi\vert_{\mathcal M}^{-1}(p) = \pi^{-1}(p) \times \T$. Such pre-images are still invariant sets for the  dynamics, as they are formed from joint level sets of averaged observables. The difference between the invariant partition corresponding to the ergodic quotient of the extended system $\xi_e$ and the invariant partition corresponding to the partial quotient $\xi\vert_{\mathcal M}$ is that only $\xi_e$ yields the ergodic partition, in general.

In plainer terms: any analysis based on the partial quotient $\xi\vert_{\mathcal M}$ will identify features distinct only in physical state-space $\mathcal M$, but not distinctions in physical time. This is sometimes desirable; as a trivial example consider a simple oscillator, $\dot \theta = 1$ for $\theta \in \mathcal M = \T$, and, even though it is already autonomous, consider its formal extension into $\mathcal M_e = \T^2$. Analysis of the partial quotient $\xi\vert_{\mathcal M}$ would identify a single invariant set $\mathcal M$, while the analysis of $\xi_e$ would identify a continuum of ergodic sets, each corresponding to a trajectory with a different initial phase. Depending on the application, either of the two analysis can contain valuable information about the analyzed system.

A common technique used in analysis of periodically-driven systems is the Poincar\'e map $\Phi_{t_0}^n:\mathcal M \to \mathcal M$ associated to the system:
\[
\Phi_{t_0}^n(x) := \phi_{t_0 + 2\pi n}(x).
\]
It is a simple result to show that for continuous flows $\phi_t$, the choice of the initial time (or phase) $t_0$ results in topologically-conjugate maps, implying that each there exists a homotopy such that each invariant set of $\Phi_{t_1}$ has a homotopy-related invariant counterpart in state space of $\Phi_{t_2}$. However, the analysis presented in this paper goes beyond topology, as we study geometry of the features in the state space.

\begin{figure}[ht]
  \centering
  \includegraphics[width=.7\linewidth]{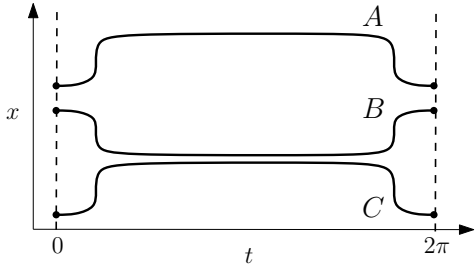}
  \caption{Three trajectories illustrating the distinction between the direct study of a Poincar\'e map and coherent features identified by partial ergodic quotient analysis.\label{fig:map-vs-eq}}
\end{figure}
Consider three trajectories sketched in the Figure \ref{fig:map-vs-eq}. If a Poincar\'e map was taken at $t \equiv 0 \pmod{2\pi}$ all of them would be represented as fixed points in state space of the map. Looking at just the Poincar\'e map, we would be much more likely to identify trajectories $A$ and $B$ as similar, since they pierce the Poincar\'e surface at nearby locations. However, considering the entire evolution, trajectories $B$ and $C$ deviate significantly only over relatively short time interval. Since time averages capture mean behavior, $H^{-s}$ distance on $\pi\vert_{\mathcal M}$ would (correctly) identify $B$ and $C$ as similar, regardless of the time instance at which Poincar\'e section used for visualization is taken.

In summary, constructing the partial quotient, i.e., embedding the state space using trajectory averages of observables on the ``physical'' states  $\mathcal M$, identifies ergodic partition of the Poincar\'e map. Additionally, the metric structure on the partial quotient reveals which invariant sets are similar based on the \emph{entire trajectory evolution}, not just on the points at which trajectories pierce the Poincar\'e surface. Therefore, this method extends the Poincar\'e map analysis by providing additional information which is potentially useful in identifying coherent structures. To demonstrate, we proceed with a numerical example.

\subsection{Example: Periodically-driven 3D Hill's Vortex Flow}
\label{sec:hills-vortex}

\begin{figure*}[htb]
  \centering
\subfloat[The slice at $\theta=0$ through the Poincar\'e section at $t \equiv 0\pmod{1}$ in the invariant region; color interpolated using the first diffusion eigenfunction $\chi_1$ from points in panel \protect\subref{fig:hill-lopert-points}.] {\label{fig:hill-lopert-section}
\includegraphics[height=60mm]{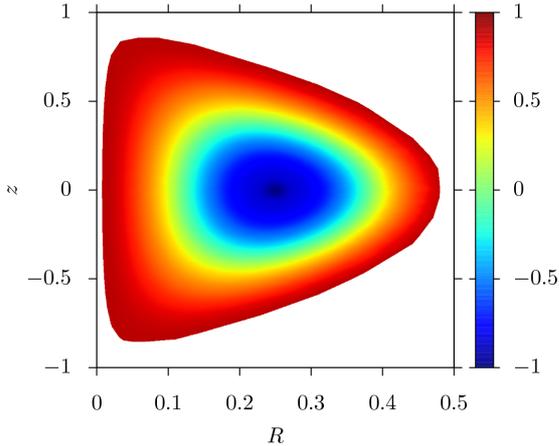}}
\hspace{.1\linewidth}
  \subfloat[Points of the Poincar\'e section used for interpolation in panel \protect\subref{fig:hill-lopert-section}, colored using the first diffusion eigenfunction $\chi_1$.]{\label{fig:hill-lopert-points}
\includegraphics[height=60mm]{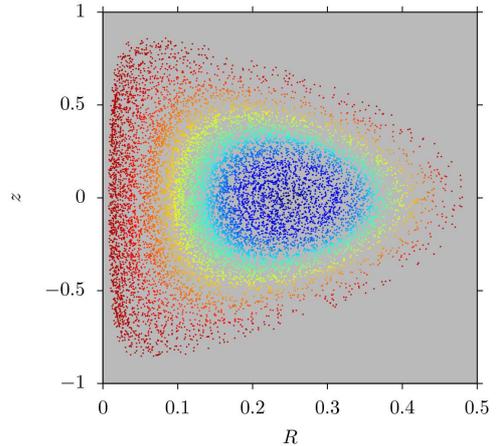}}\\
\subfloat[Bounded level sets of the unperturbed Hamiltonian $H(R,z) = Rz^2 - R +2R^2$ in a $(R,z)$-plane.]{\label{fig:hill-hamiltonian}
\includegraphics[height=45mm]{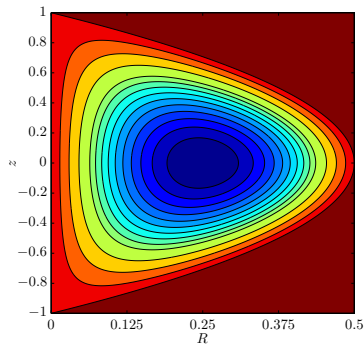}}
\hspace{.1\linewidth}
\subfloat[Diffusion eigenfunctions $\chi_k$, $k=2,3,4,5$ of the ergodic quotient plotted against  $\chi_1$.]
{\label{fig:hill-lopert-eq}
\includegraphics[height=45mm]{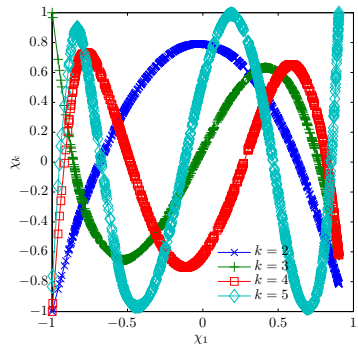}}
  \caption{Structure of the invariant sets in the state space of a periodically-forced 3D Hill's vortex at low perturbation $c=\epsilon=0.01$.}
  \label{fig:hill-lopert}
\end{figure*}

As an example of a periodically-driven flow, consider the following ODE,  evolving states $x = (R,z,\theta) \subset \R^+ \times \R \times \T$:
\[
\begin{bmatrix}
  \dot R \\ \dot z \\ \dot \theta
\end{bmatrix} = \underbrace{
\begin{bmatrix}
  2Rz \\ 1 - 4R - z^2 \\ c/2R
\end{bmatrix}
+ \epsilon
\begin{bmatrix}
  \sqrt{2R} \sin \theta \\
  z(\sqrt{2R})^{-1} \sin\theta \\
  2\cos \theta
\end{bmatrix}
\sin 2\pi t}_{A(x,t)},
\]
where parameters  $c$ and $\epsilon$ are swirl and perturbation strengths, respectively. For all values of $c$ and $\epsilon$ the system preserves the state space volume, i.e., $\nabla \cdot A \equiv 0$, where $\nabla = (\partial_R, \partial_z, R^{-1}\partial_\theta)$. This system was studied analytically in \cite{Mezic:1994ez,Vaidya:2011wo} as an example of a non-hamiltonian system showing Kolmogorov-Arnold-Moser-type structures at low values of $\epsilon=c$. In this paper, we will show the structure of invariant sets for both low- and high-valued perturbations for the purposes of illustrating the method of analysis; therefore, we will not explore here the bifurcation mechanism behind the structures in the state space. A preliminary numerical study was presented in \cite{Budisic:2011ky}.

The results of the first simulation (Figure \ref{fig:hill-lopert}) demonstrate the low-perturbation structure of the state space, at $c=\epsilon=0.01$. Uniformly distributed initial conditions ($N=1000$) were seeded on the $\theta=0$ plane in the rectangle $(R,z) \in [0.01,0.30] \times [-0.5,0.5]$, which is within a bounded invariant set. The observables were chosen as Fourier basis functions on $[0.0,0.5] \times [-1.0,1.0] \times [0,2\pi] \subset \mathcal M$, with wavevectors in the box $\abs{k_R} \leq 8$, $\abs{k_z} \leq 8$, $\abs{k_\theta} \leq 4$. Simulations were run for no less than $T_{min} = 600$, and until tolerance $\text{ATOL}=10^{-4}$ was achieved (see \ref{sec:convergence} for definition), with most simulations terminating before $2000$ time units expired. 

At $c = \epsilon = 0$, any fixed-$\theta$ plane is invariant and on it the system conserves the Hamiltonian $H(R,z) = Rz^2 - R +2R^2$. The Figure \ref{fig:hill-lopert} demonstrates that at low perturbation values, the structure of the cross-section of invariant sets is very similar to the level sets of the Hamiltonian, despite the flow having a full three-dimensional character. At the plotted scale, $\chi_1$ appears to be a monotonic function of the Hamiltonian; however, the non-hamiltonian KAM theory predicts that invariant tori do not form a continuum, but are interspersed with thin chaotic zones. Observables of low wavenumbers effectively coarse-grain dynamics; since the chaotic zones are so thin at low perturbation values, they are not identified as significantly different invariant sets. Plotting higher-order diffusion eigenfunctions against the $\chi_1$ reveals the similarity of diffusion eigenfunctions to the Legendre polynomial system, which is likely due to connection of Legendre polynomials to spherical harmonics and the spherical shape of the Hamiltonian function. In this case it is particularly obvious how monotonic parametrization of the ergodic quotient can serve as a proxy for conserved quantities, possibly leading to Fomenko-type analysis.

\begin{figure*}[htb]
  \centering
  \subfloat[The Poincar\'e section at $t \equiv 0 \pmod{1}$. Segments of state space corresponding to clusters $0$ and $3$ shown in blue and red, respectively. The remaining clusters are cropped for clarity. ]{\label{fig:hill-hipert-iso}
\includegraphics[height=50mm]{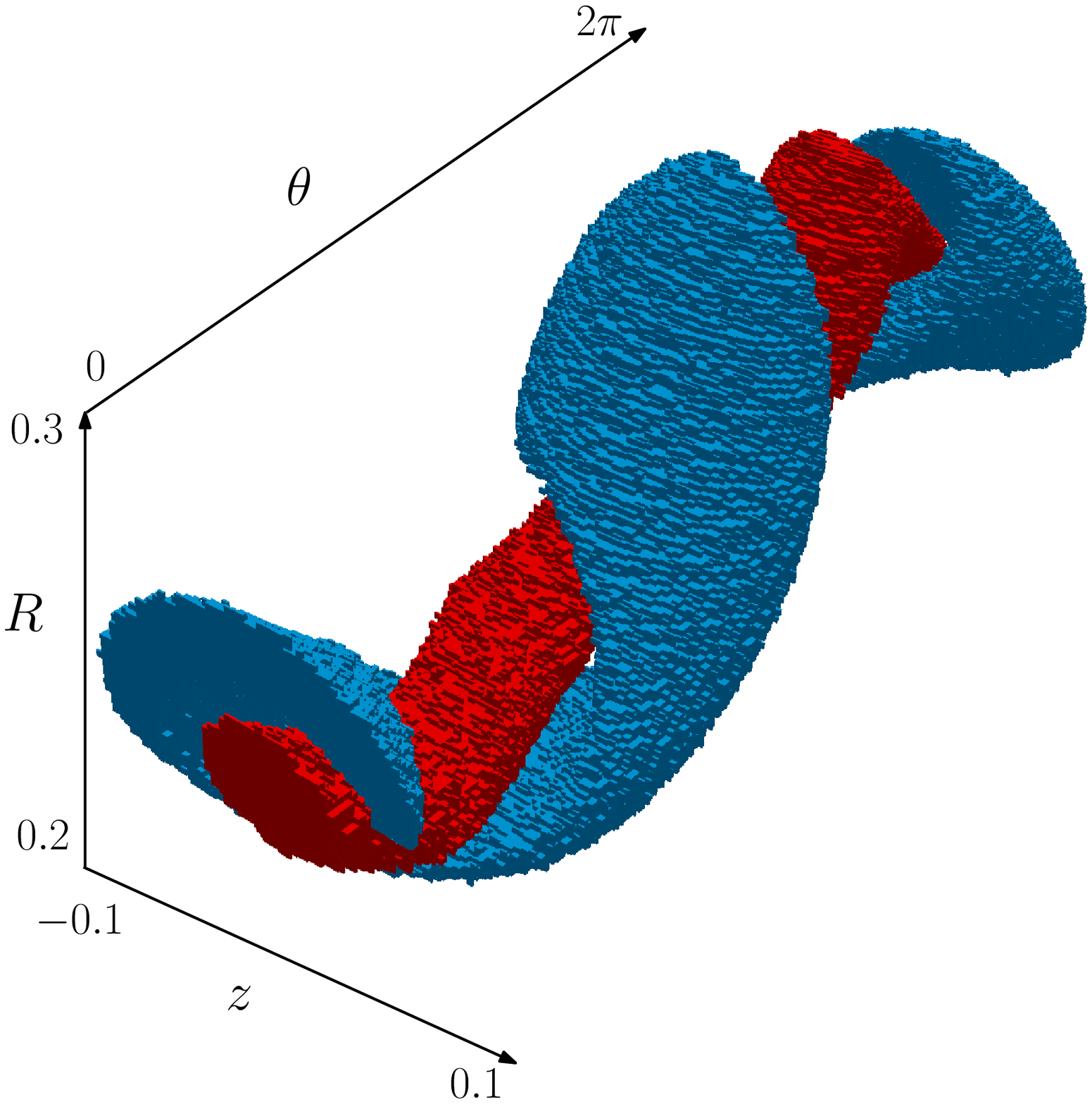}}
\hspace{.1\linewidth}  
\subfloat[The slice at $\theta=0$ through the Poincar\'e section in panel \protect\subref{fig:hill-hipert-iso}, displaying all clusters; colors reflect cluster labels ($0,1,2,3$).]{\label{fig:hill-hipert-slice}
\includegraphics[height=50mm]{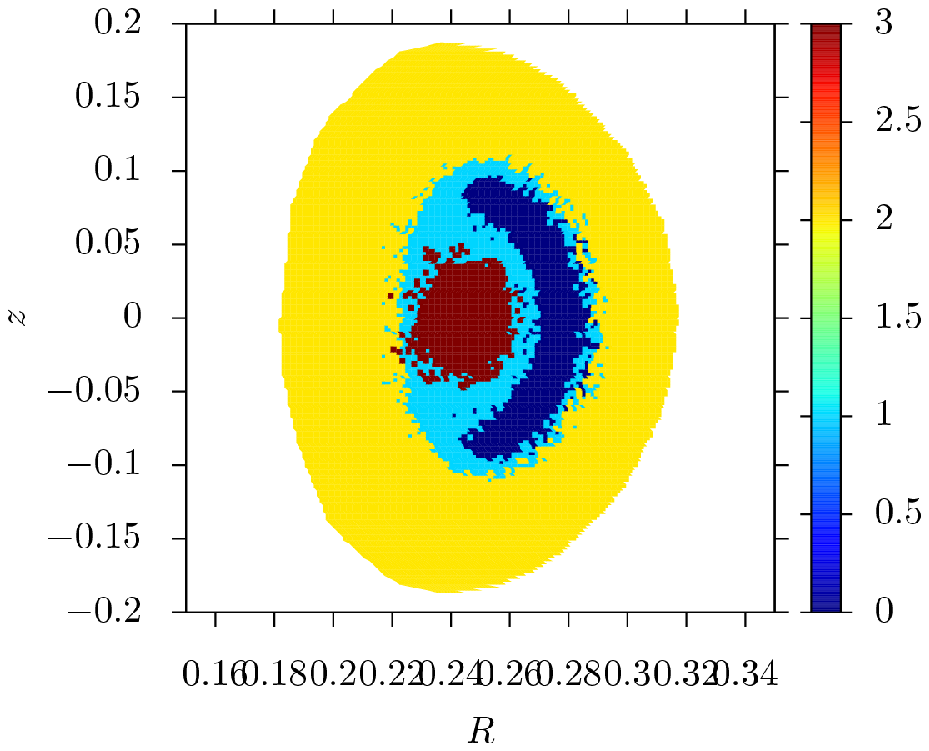}}
 \hfill \\
  \subfloat[Embedding of the partial quotient into first two diffusion eigenfunctions $\chi_1$, $\chi_2$,  with color indicating the cluster assigned to the point.]
{\label{fig:hill-hipert-eq}
\includegraphics[height=50mm]{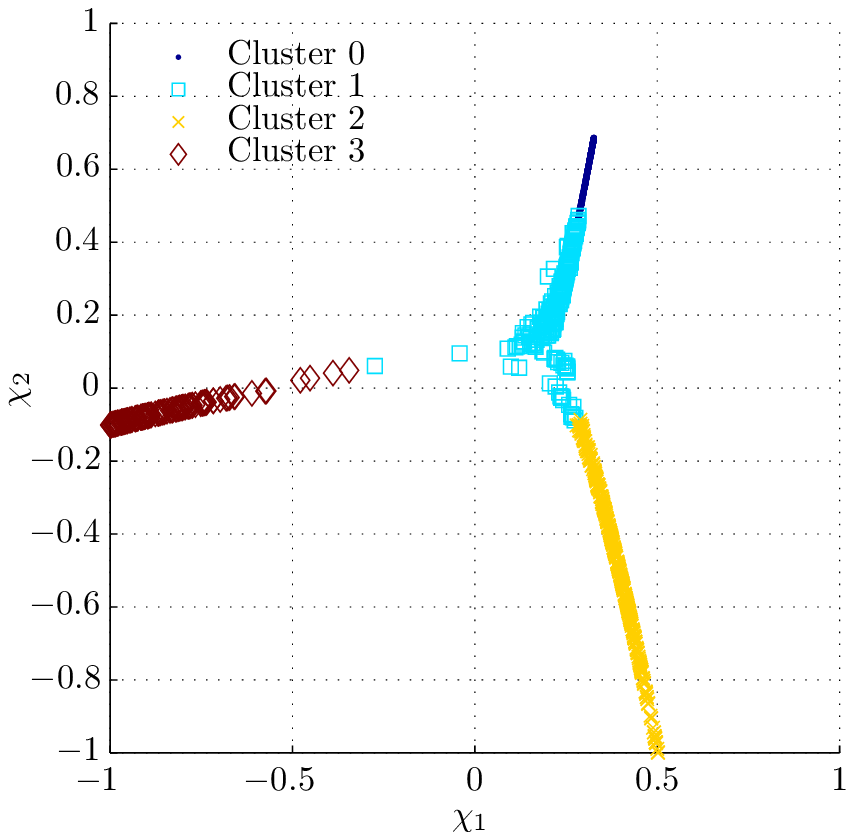}}
\hspace{.1\linewidth}
  \subfloat[Diffusion eigenfunctions $\chi_k$ of the partial quotient, plotted against  $\chi_1$ for points in the cluster $2$]{\label{fig:hill-hipert-diff}
\includegraphics[height=50mm]{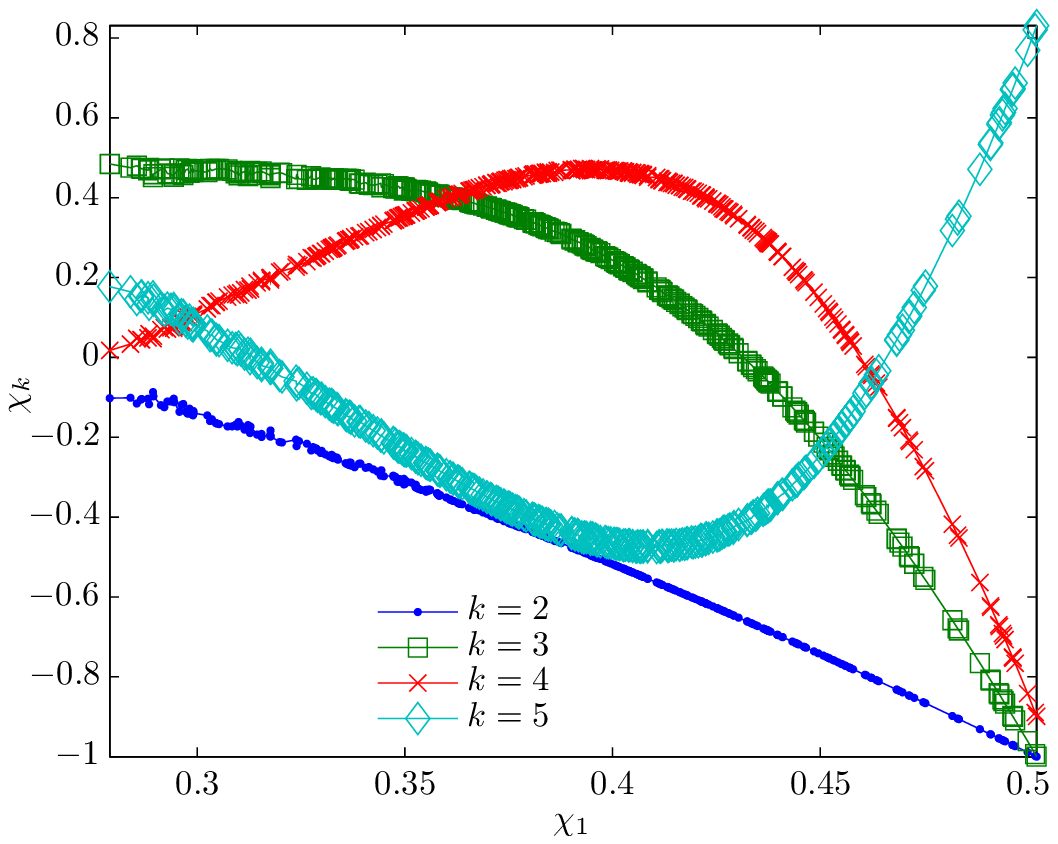}} 
\hfill
  \caption{Structure of the invariant sets of a periodically-forced 3D Hill's vortex at high perturbation $c = \epsilon = 0.3495$. Diffusion coordinates were clustered using $k$-means algorithm into $4$ clusters to extract the main features. The Poincar\'e section used for visualization is at $t \equiv 0 \pmod{1}.$}
  \label{fig:hill-hipert}
\end{figure*}

When the perturbation is increased, the dissolution of invariant tori continues, but an additional invariant set becomes visible in the state space. For the next simulation, we have chosen $c = \epsilon = 0.3495$. There is nothing special about this particular parameter value: all perturbations $c=\epsilon$ between roughly $0.2$ and $0.4$ exhibit the same topological structure, visible on similar scales. As the region where trajectories stay bounded shrinks with the growth of the perturbation, we have seeded the initial conditions uniformly ($N = 1000$) in $(R,z) \in [0.2,0.3] \times [-0.1,0.1]$ in $\theta=0$ plane, with initial time $t=0$. The observables were Fourier functions on $[0.0,0.5] \times [-0.5,0.5] \times [0,2\pi)$, with the same wave-vector bounds as before. The simulation times and tolerances also remained as in the low-perturbation case.

The Figures \ref{fig:hill-hipert-iso} and \ref{fig:hill-hipert-slice} show a secondary coherent set that becomes visible within the onion layers at higher perturbation values. The primary tubular invariant set (in red, Cluster $3$) remains at the center, with the secondary invariant set (in blue, Cluster $0$), sickle shaped in its cross section, wrapping around the tubular core in $1$:$2$ resonance. Both primary and secondary sets are enclosed by the outer tubular shells (in yellow, Cluster $2$). Here, we stress again, that even though we are showing sets in a three-dimensional Poincar\`e section, they are truly time invariant, despite the time-dependent forcing. More precisely, to each of the invariant sets $C \subset \mathcal M$ identified by clustering, there exists a direct extension in the extended state space $C \times \T \subset \mathcal M_e$.

The topology of the embedded quotient reflects the change in the invariant set structure, compared to the low perturbation case. While embedding in $\chi_1$ vs.\ $\chi_2$ in Figure \ref{fig:hill-lopert-eq} showed a topological line interval, the high perturbation case in Figure \ref{fig:hill-hipert-eq} shows that a branch splits off the interval which was once formed by segments corresponding to clusters that compose the inner and outer tubular regions.  The Figure \ref{fig:hill-hipert-diff} takes a closer look at diffusion eigenfunctions within only one of the identified clusters. Comparing the shapes of coordinate functions with the low perturbation case in Figure \ref{fig:hill-lopert-eq} in interval $\chi_1 \in [0.4, 1]$, we can again notice the similarity to Legendre polynomials, indicating that the flow appears to locally conserve a continuous motion integral, even if that is not the case globally.

\subsection{Harmonic quotients}
\label{sec:harmonicquotient}

While this paper deals primarily with time-invariant structures, the Section \ref{sec:periodically-driven} describes how even time-varying flows can be analyzed using trajectory averages, through embedding of the system into an extended state space, expanded by time variable. Fourier observables $g$ on the extended state space $\mathcal M_e = \mathcal M \times \mathcal T$ were formed from Fourier observables $f_k : \mathcal M \to \C$ as $g_w(x,t) = f_k(x) e^{i2\pi \omega t}$, where $w = (k,\omega)$ was the extended wave-vector, with $\omega$ the wavenumber along time axis. Time averages of observables were given by
\begin{align*}
  \tilde g_w(x,t) &= \lim_{T \to \infty} \frac{1}{T} \int_0^T g[ \phi_\tau(x), \tau] d\tau \\
  &=\lim_{T \to \infty} \frac{1}{T} \int_0^T [f \circ \phi_\tau](x) e^{i2\pi\omega \tau} d\tau.
\end{align*}
When $\mathcal T = \T$, it was sufficient to consider only $\omega \in \Z$ to obtain a full function basis on $\mathcal M_e$. However, when the time is the full real line, as in the aperiodic case, we require a real-valued wavenumber $\omega \in \R$. The complete functional basis on $\mathcal M_e$ is then uncountable, and the ergodic quotient of the extended system, evolving in $\mathcal M_e$, would not be a subset of a sequence space anymore.

Nevertheless, even in the case of flows $\phi_t$ generated by autonomous systems, it is of interest to fix $\omega$ and to consider functions $f^{(\omega)}$ given by averages
\begin{align*}
  \tilde f^{(\omega)}(x) \triangleq \lim_{T \to \infty} \frac{1}{T}
  \int_0^T [f \circ \phi_\tau](x) e^{i2\pi\omega \tau} d\tau,\label{eq:harm-avg}
\end{align*}
termed \emph{harmonic averages}. The harmonic averages were considered by Wiener and Wintner under the name \emph{Fourier averages} \cite{Wiener:1941wy}, but more recently, they have been analyzed in \cite{Mezic:2004is}. We use them here to provide a generalization of the ergodic quotients which can be used to analyze periodic and quasi-periodic transport.

For any fixed $\omega \in \R$, and Fourier observables $f_k$, the embedding $\pi_\omega(x) = (\dots, \tilde f_k^{(\omega)}(x), \dots)_{k \in \Z^D}$
generates an analog of the ergodic quotient, the \emph{harmonic quotient at frequency $\omega$}. The ergodic quotient is then a special case of the harmonic quotient, computed for $\omega = 0$. In understanding the ergodic quotient, the level sets of averaged observables $\tilde f$ played a crucial role as invariant sets. Similarly, level sets of $\tilde f^{(\omega)}$ provide insight into periodic sets and periodic transport when $\omega \in \Q$, and wandering sets and quasi-periodic transport when $\omega \not \in \Q$ \cite{Mezic:2004is}. More concretely, if $\omega = 1/2$, $\tilde f^{(\omega)}$ will vanish over all points $x$ that are not $\phi_t$-periodic with period $2$. Consequently, the joint level sets of $\tilde f^{(\omega)}_k$ away from $(0,0,\dots)$ identify sets in state space between which there exists dynamical transport of period $2$.

The functional mapping $f \mapsto \tilde f^{(\omega)}$ has an interpretation connected to the Koopman operator. For systems satisfying conditions in this paper, the Koopman group consists of linear, bounded composition operators $U_t : \mathcal F \to \mathcal F$, defined as
\begin{align*}
  U_t f \triangleq f \circ \phi_t,
\end{align*}
for time $t \in \mathcal T$. The space $\mathcal F$ is a space of observables, in literature typically chosen as $L^\infty(\mathcal M, \mu)$, where $\mu$ is a $\phi_t$-invariant measure. The Koopman operator is the generator of the Lie group; when time is discrete, i.e.,  $\mathcal T \subset \Z$, the generator is often denoted simply by $U$.

The harmonic averages $\tilde f^{(\omega)}$ are eigenfunctions of the Koopman operator, since it is simple to show that \[ U_t \tilde f^{(\omega)} = e^{i2\pi \omega t} \tilde f^{(\omega)}.\] It follows that, for any observable $f$, the harmonic average $\tilde f^{(\omega)}$ is an eigenfunction of the Koopman group generator at eigenvalue $\lambda_\omega = e^{i2\pi\omega}$ for discrete time, and at eigenvalue $\lambda_\omega = {i2\pi\omega}$ for continuous time. The operator $P_\omega : \mathcal F \to \mathcal F$,
\[
P_\omega f = \tilde f^{(\omega)}
\]
is a projection operator, projecting the space of observables onto the  eigenspace of the Koopman operator at $\lambda_\omega$. In this context, the harmonic quotient map $\pi_\omega$, applied to the Fourier basis, evaluates the functions spanning the eigenspace at $\lambda_\omega$. Therefore, by endowing the harmonic quotient with a particular geometry, we are able to computationally analyze the eigenspaces of the Koopman operator, without ever computing, or approximating, the Koopman operator itself.

\section{Discussion}
\label{sec:discussion}

It is difficult to give a level-field comparison of different methods for identification of ``coherent structures'' in state spaces due to the lack of a method-independent, universal definition of what a coherent structure might be. While each of the methods, namely, LCS, Ulam's approximation, and ergodic quotient, has an operational understanding of a coherent structure, ultimately the experience and application is the only judge of how well a method captures what is identified as ``coherent'' in practical settings. For this reason, we decided to apply the analysis to the ABC flow; this was not because we viewed the ABC flow as a particularly challenging problem, but rather because it was well analyzed both theoretically and computationally, resulting in well-established knowledge about what coherent structures one should obtain. In this sense, we found that the analysis of ergodic quotient identified the same structures as the theoretical analysis and the alternative computational methods.

When it comes to the amount of work, computational or experimental, needed to compute an approximation to invariant sets, one has to asses the setting in which either of the methods is applied. An advantage of the ergodic quotient method is that it requires a single initial condition to be placed in any ergodic set sought to be identified. In practical terms, this means that simulations or experiments need to be initialized only on a small subset of the entire state space: this was clearly exemplified in Section \ref{sec:hills-vortex} where simulations were initialized on the $\theta = 0$ planes and the method successfully identified coherent structures in the full $(R,z,\theta)$ state space. By contrast, methods based on Ulam's approximation require seeding the initial conditions in the entire invariant domain. While this might not be a significant distinction when working with ODEs, in experimental settings, e.g., when analyzing fluid flows using Particle Tracking Velocimetry, it might be possible to initialize the trajectories only in a specific, small region, and count on the long experimental runs to collect data about the rest of the state space. Such data is a natural input to ergodic quotient methods, whereas Ulam-type computations would require additional processing, e.g., flow interpolation, that might introduce further errors.

The number of initial conditions required to approximate the ergodic quotient is fairly low: only $N=1000$ trajectories were used to retrieve a very detailed image ($S = 200^3$ spatial cells) of the flow structures in Section \ref{sec:abcflow}. This is a consequence of the recognition that the entire trajectory can be described to a high precision by a relatively small set of values, i.e., averages of a function basis. While the entire trajectory evolution is used for \emph{visualization} of the structures, the \emph{identification} of structures, where the algorithmic effort is really spent, is performed only with $N$ data-points (samples of the ergodic quotient) as inputs. 

In terms of numerical linear algebra, we do acknowledge that although the Ulam's matrix is large, $S \times S$, where $S$ is the total number of spatial cells, it is often sparse. In contrast, the eigenproblem at the core of the Diffusion Maps is is potentially dense, although on a much smaller $N \times N$ matrix, where $N$ is the number of trajectories.\footnote{Due to the rapidly-decaying exponential diffusion kernel, however, the diffusion eigenproblem might also be sparsified by truncating the kernel. We have not explored the consequences of sparsification, though.} Nevertheless, we see the distinction as conceptual, rather than computational, signifying that the complexity of the state space is often captured by a far smaller number of parameters than needed to describe transport between pairs of cells in state space which is especially true for (near-) integrable, high-dimensional systems. This paper focuses on visualization as the first step in validation of the ergodic quotient approach; however, we expect that the ergodic quotient could be analyzed directly, for example, to detect topological changes such as those between embedding $\chi_1$ vs.\ $\chi_2$ in Figures \ref{fig:hill-lopert-eq} and \ref{fig:hill-hipert-eq}. These changes indicate a bifurcation of the flow, without ever visualizing the state space. If our expectations are confirmed, the high-dimensional state spaces, which cannot be visualized and requiring a large number of cells to discretize the state space, could potentially be analyzed through the study of a small number of trajectories by noting topological changes in the ergodic quotient.

Both the Ulam's approximation and the ergodic quotient construction feature exponential growth of problem size with the increase of the state space dimension $D$. Number of state space discretization cells, for a fixed axial resolution $R$, grows as $R^D$. Similarly, to approximate the ergodic quotient, if we truncate the function basis at a uniform cutoff wavenumber $K$, the number of required Fourier observables is $K^D$. For certain high dimensional systems, it is not always necessary to analyze the state space directly: rather a particular, low-dimensional output space might suffice. For example, in networks of power generators the variable of interest in detecting instabilities might be the mean frequency across all generators (see \cite{Susuki:2011jz}). If instead of the state space we study observables only on the output space, the number of required observables can be significantly lower. It is therefore possible that it is sufficient to initialize the trajectories on a low-dimensional set in the state space and analyze it in a low-dimensional output space, despite the dynamics being high-dimensional. Both order reductions are a natural fit with the presented technique in analysis of trajectory averages presented here.

The application of the ergodic quotient techniques to periodically-forced systems has been presented in Section \ref{sec:time-varying}. The extension of the method to the case of transient systems, where long-time averages lose the information about the initial evolution, is the focus of our future studies. The case of transient systems is related to the study of finite-time systems, i.e., those defined only on a fixed time interval, which includes all data sets where experiment time cannot be extended into the steady regime, e.g., field samples of geostrophic flows. We do not intend delve deep into a discussion about how much we can predict about the future of an aperiodic flow based only on a finite time interval. Instead, we point out that the analysis presented in this paper does have a very clear interpretation even on a fixed finite-time interval: the finite-time quotient would reflect sets of initial conditions that remain together under advection by the flow map over the interval considered. In that case, the $H^{-s}$ metric and diffusion maps result in a multi-scale aggregation of finite-time trajectories, identifying regions of space where material is transported in a homogeneous fashion. The resulting ``coherent structures'' present a template of the transport of material over the analyzed time interval. It is left for future research to study the relationship between transport templates corresponding to  overlapping or adjoining time intervals and their invariance and predictive power.

Recently, the analysis of mesohyperbolicity and mesoellipticity has yielded success in the study of finite-time flows. Both the ergodic quotient and mesohyperbolicity depend on averages of observables along trajectory lines. The first distinction between these methods is in their purpose: mesohyperbolicity and mesoellipticity are proposed generalizations of concepts of hyperbolicity and ellipticity from autonomous, steady-state flows to unsteady, finite-time flows. In contrast, ergodic quotient is directly related to the ergodic partition, which studies invariant sets. The two concepts are potentially related, as is the case in classical theory of autonomous systems, where the relationship between analysis of invariant manifolds and analysis of Lyapunov coefficients is standard. 

The second, more technical, distinction is in the objects averaged along trajectories. In construction of the ergodic quotient, the ultimate goal is to include all observables in the averaging process, in an attempt to represent the underlying empirical measures as well as it is possible. The observables chosen depend only on the topology of the state space, not on the actual flow map of the dynamical system. In studies of mesohyperbolicity/ellipticity, a very particular set of observables is chosen: they are entries in the Jacobian matrix of the flow. Therefore, while the  methods use the same technique of computing averages of observables, the theory behind the averages serves different purposes, and, ultimately, demonstrates the versatility of information that can be extracted from trajectory averages of observables.

From the standpoint of the direct application of methods presented in this paper, there are still points where one could improve the presented procedure, although most of such efforts fall outside the scope of dynamical systems, e.g., into fields of machine learning or computational geometry. For example, setting the bandwidth of the exponential kernel in the Diffusion Maps algorithm (see \ref{sec:diffusionmaps}), although relatively simple to perform manually, is still not fully automatic, which might slow down the analysis process when dealing with a particularly complex ergodic quotient. 


Furthermore, we were guided mostly by simplicity and clarity in our choice of the $k$-means algorithm for post-processing of diffusion coordinate, in an effort not to derail the main point of this paper. The Figure \ref{fig:hill-hipert-eq} reflected the imperfections in the choice of $k$-means clustering as the final step in post-processing. One could expect that the four clusters identified would be the three ``branches'' and the point at which they merge. While the numerical results come close to such classification of points, there is room for improvement. It is possible that instead of identifying clusters, identification of connected components or one-dimensional sets, based on a proximity graph in a diffusion coordinate space, would yield even better results than those presented here.  

\section{Conclusion}
\label{sec:conclusion}

Our results demonstrate that, by analyzing the ergodic quotient, we can infer how trajectories organize into coherent structures in the state space. To construct an approximation to the ergodic quotient, we average a subset of a function basis on the state space, and endow it with the empirical distance. In this setting, the geometry of the ergodic quotient reveals dynamically coherent regions in the state space. The geometry was studied using diffusion coordinates which rectify the ergodic quotient, enabling clear visualization and extraction of clusters corresponding to coherent structures.

To evaluate the method, we identified coherent structures for the ABC flow, an area\hyp{}preserving, three\hyp{}dimensional, steady flow. Our results correspond to those coherent structures identified by Lagrangian Coherent Structures \cite{Haller:2001vh} and Ulam's method \cite{Froyland:2009ti}. Unlike those approaches, however, our method does not require placing the initial conditions in the entire state space, but rather depends on long-time integration to explore the state space. For the same reason, it requires a relatively small number of initial conditions. This makes it suitable for analyzing experimentally harvested trajectories in experiments where initial conditions can be placed in a small region of the state space, but the dynamics of the system evolves on a larger domain. 

For autonomous flows, points on the ergodic quotient represent ergodic invariant sets in the state space. When the flow is periodically-dependent on time, the construction we present identifies invariant sets of the Poincar\'e map. In contrast to analysis of Poincar\'e map, however, coarse graining of invariant sets takes into the account the entire evolution of the trajectory, not only their Poincar\'e section, providing a more accurate method of identifying trajectories which have similar behavior. On an example of a periodically-forced 3D Hill's vortex flow we show that the method is able to identify previously undescribed coherent features in state spaces. Finally, we close the discussion of time-dependent features by generalizing the ergodic quotient to harmonic quotients, which can provide information about periodic and quasi-periodic transport in the state space.

Although this paper discusses applications of the ergodic quotient to analysis, we believe that the proposed framework could, in future, be useful in design. Diffusion coordinates provide an orthogonal system of invariant functions for the system. They are, therefore, particularly suitable as a setting for design of invariants of motion, e.g., through optimization of linear combinations of diffusion coordinates based on a design criterion.

\section{Acknowledgments}
  Presented research was funded by the following grants: DARPA FA9550-07-C-0024, AFOSR FA9550-09-1-141, and ONR MURI N00014-11-1-0087.

  The first author is grateful to Ryan Mohr, whose useful comments
  improved this paper. The authors additionally thank the anonymous reviewers for their suggestions.

  The code for analysis of maps and ODEs was implemented in
  Python programming language, using Scipy
  \cite{jones01:scipy}, PyDSTool \cite{clewley04:pydstool}
  (automatic generation of Radau numerical integrators), and PyTables
  libraries \cite{alted02:pytables} (HDF5 data storage). 

\appendix

\section{Convergence of averages}
\label{sec:convergence}

Our implementation of the averaging process computes averages of observables along with the numerical evolution of the trajectory, i.e., ``on-line''. A heuristic criterion uses the on-line finite averages to terminate the trajectory integration once a desired tolerance on convergence is achieved. As a result, different initial conditions will result in different lengths of trajectories, depending on the regularity of trajectory that evolves from an initial condition.

The ODE model of a system is numerically integrated producing state vectors $\phi_{t_n}(x)$ at adaptively-determined time instances $t_n$. Summing a  zeroth-order interpolation of evolution of $f_k$ approximates the continuous-time average along the trajectory:
\begin{align*}
  \frac{1}{t_N - t_0}&\int_{t_0}^{t_N} \left[f \circ \phi_t \right](x) dt \approx \notag\\
  \frac{1}{t_N - t_0}&\sum_{n=1}^N (t_n - t_{n-1}) [f \circ \phi_{t_{n-1}}](x)
\end{align*}

To determine the stopping time of the integration, we pause the integrator in intervals of fixed length $T_e$ and at every time instance $t_0 + j T_e$ compare the vector of time averages with the vector obtained at the previous pause $t_0 + (j-1)T_e$:
\begin{align*}
  \text{ADIFF}_j(x) &= \norm{ \tilde F_j(x) - \tilde F_{j-1}(x) }_\infty
\end{align*}
where $\tilde F_j(x)$ indicates the vector containing all the chosen observables averaged along trajectory starting at $x$, computed in the $[t_0,t_0 + jT_e]$ interval. When $\text{ADIFF}_j(x) < \text{ATOL}$ for some pre-specified small tolerance $\text{ATOL}$, the integration is stopped, and $F_j(x)$ taken to be the value of time\hyp{}averaged observables at point $x$. The underlying assumption is that the finite-time averages will show little variation as the infinite limit is approached. As an alternative to $\norm{.}_\infty$, the $H^{-s}$ distance \eqref{eq:negsobdist} could have been used in the stopping criterion, however, we have not explored that option here.

Theory gives some guidance on times required for convergence of finite-time averages: over regular trajectories, averages converge with $T^{-1}$, while over mixing trajectories, the slope is gentler, $T^{-1/2}$. A classical result asserts that for other types of trajectories the slope $T^{-\alpha}$ can be arbitrarily small \cite[][\S 3.2B]{Petersen:1989ly}. A computational study of convergence with various slopes is given in \cite{Levnajic:2010gq}. From a practical standpoint, the outlook is not so bleak: slopes of small $\alpha$ are to be typically expected to occur along trajectories that pass close to homoclinic and heteroclinic orbits, and get entrained around periodic orbits that exist in their vicinity (intermittency, or stickiness) \cite{Perry:1994wl}. Regions of intermittency are typically small, e.g., for a pendulum perturbed by a small periodic forcing of magnitude $\epsilon$, the area of intermittent regions scales as $\epsilon \log (1/\epsilon)$ \cite{Treschev:1998uq}.

In Figure \ref{fig:convergence} we present several metrics which illustrate how well the time averages converge at the time the simulation is stopped by the $\text{ADIFF}(x) < \text{ATOL}$ criterion. Time average vectors $\hat F_\infty(x)$ and $\hat F_{\text{ATOL}}(x)$ both represent finite-time averages, computed for trajectories originating at initial conditions $x$. The difference between them is in the moment of termination of the averaging process. The vector $\hat F_{\text{ATOL}}(x)$ was obtained by stopping the averaging at time $T(x)$, dependent on the initial condition, when the described stopping criterion reached $\text{ATOL}=10^{-5}$. The vector $\hat F_{\infty}(x)$ mimics the ``true'' infinity, where averaging was terminated at the initial\hyp{}condition\hyp{}independent time $T_{\infty} = 2.5 \times 10^4$. 
We compared these finite-time vectors in the $L^\infty$ norm $\norm{.}_\infty$ (Figure \ref{fig:conv_errorinf_log}) and in the $H^{-s}$ norm $\norm{.}_{2,-s}$ \eqref{eq:negsobdist} (Figure \ref{fig:conv_erroremp_log}), as the latter is the relevant error for the remainder of the computation. 

Generated images demonstrate that the described termination criterion results in stopping times that vary by an order of magnitude (Figure \ref{fig:conv_time}). As a result, the errors in convergence are consistently low, for regular, chaotic and intermittent trajectories. If we compare the magnitude of errors in $H^{-s}$- and $L^\infty$-norms, it is clear that the use of $H^{-s}$ norm, which discounts higher wavenumbers, has a desired side-effect of increasing the robustness of the averaging to finite stopping times.

\begin{figure*}[htpb]
  \centering
  \subfloat[Simulation termination time $T(x)$]{
    \includegraphics[height=45mm]{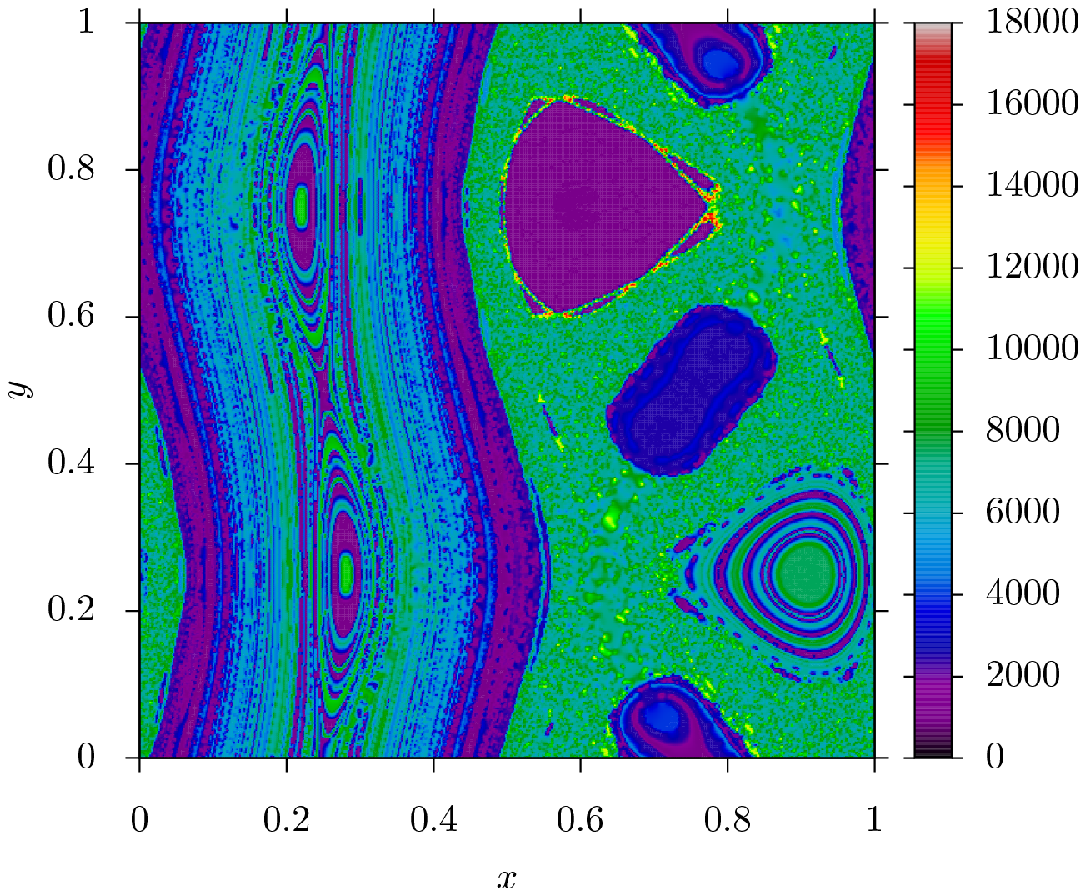}
    \label{fig:conv_time}
  }\hfill
  \subfloat[Time average error in infinity norm: $\log_{10}\norm{\hat F_{\text{ATOL}}(x) - \hat F_\infty(x)}_\infty$]{
    \includegraphics[height=45mm]{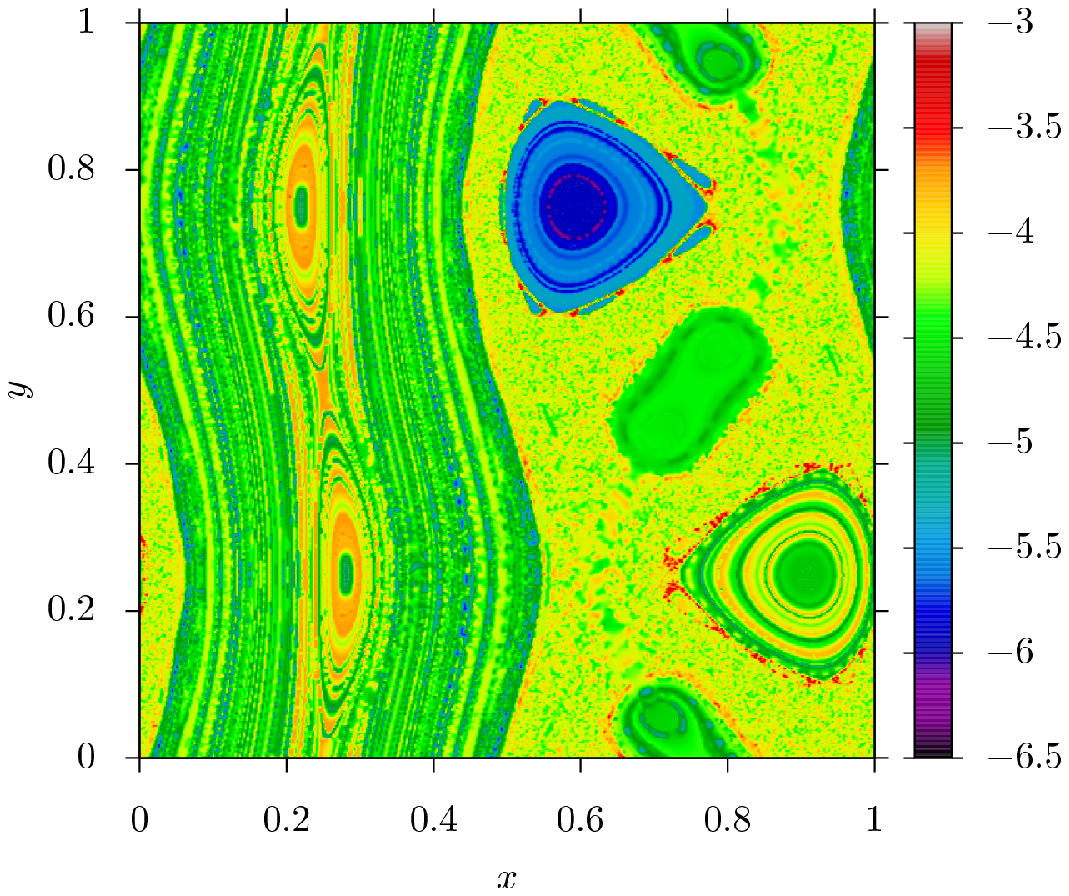}
    \label{fig:conv_errorinf_log}
  }\hfill
  \subfloat[Time average error in $H^{-s}$ norm: $\log_{10} \norm{\hat F_{\text{ATOL}}(x) - \hat F_\infty(x)}_{2,-s}$]{
    \includegraphics[height=45mm]{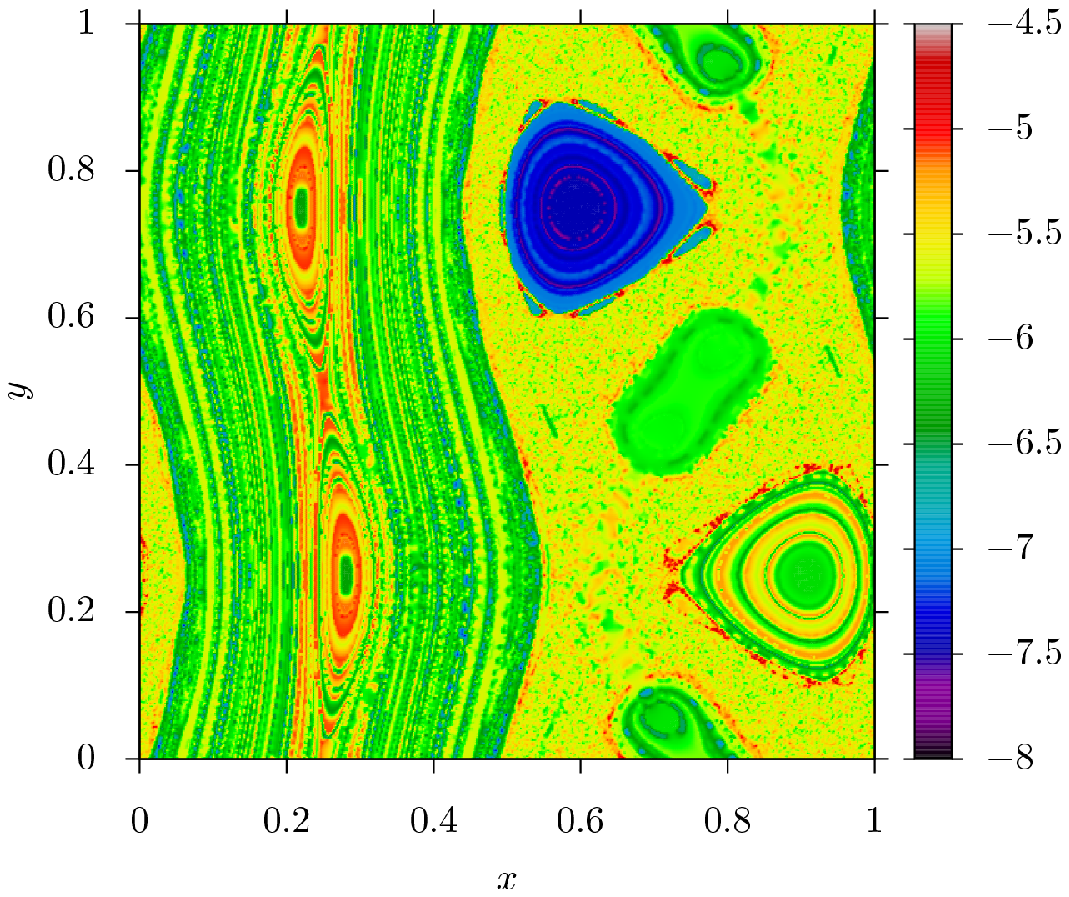}
    \label{fig:conv_erroremp_log}
  }
  \caption{Convergence of averages for $N=500$ trajectories of ABC flow starting at plane $z=0$. Simulations were terminated at tolerance $\text{ATOL}=1 \times 10^{-5}$. 
Figures show: 
\ref{fig:conv_time} the termination time $T(x)$, 
\ref{fig:conv_errorinf_log},  
\ref{fig:conv_erroremp_log} 
    the difference between time averages at time $T(x)$ and at ``true'' infinity  $T_\infty=2.5 \times 10^4$, in infinity and $H^{-s}$ norms, respectively. Observables were Fourier harmonics \eqref{eq:fourierfunction} with  wavevectors $k \in [-3,3]^3$, and simulation extension time was $T_e=500$.}
\label{fig:convergence}
\end{figure*}

Even though the proposed criterion is only one of many criteria which could be used for stopping, we believe that, based on the results presented, it provides a good heuristic that consistently produced low errors in time averages, while keeping the simulation times reasonable.

\section{A finite approximation to empirical metric}
\label{sec:finitewavevector}

A finite approximation to the norm on the Sobolev space $H^{-s}$ \eqref{eq:negsobdist} is computed by truncating the infinite sum over wavevectors $k \in \Z^D$ to a box $k \in [-K,K]^D$. In this appendix we show that such approximation converges with rate $\sqrt{K}$ with the rate constant depending on the dimension of the system.

\begin{prop}[Error in finite-wavenumber approximation to $H^{-s}$ norm]
  Let the error in approximation $\epsilon_K$ be given by a weighted sum of coefficients, $c_k = \abs{\tilde f_k(x) - \tilde f_k(y)}$, summed outside the $[-K,K]^D$ box:
  \begin{equation*}
    \epsilon_K^2 = \sum_{k \not \in [-K,K]^D} \frac{ c_k^2 }{ \left[ 1 + \left(2\pi \norm{k}_2 \right)^2\right]^s },
  \end{equation*}
  where $s = (D+1)/2$, and $0 \leq c_k \leq \mathcal C(D)$ are differences in averages of Fourier observables scaled to unit $L_2$ norm.
  Then 
  \begin{align*}
      \epsilon_K < \frac{ \mathcal E(D) }{ \sqrt{K} }, \quad \text{with} \quad
      \mathcal E(D) < \left(\frac{2}{\pi}\right)^{3/4} \frac{ D^{1/4} }{\pi^{D}},
  \end{align*}
  where $\mathcal{E}(D)$ is the error decay rate constant.
\end{prop}
\begin{proof}
Norms on $\Z^D$ are ordered in scale,  $\norm{k}_\infty < \norm{k}_2$ on $\Z^D$, enabling us to bound $\epsilon_K$ by summing over $\infty$-spheres in $\Z^D$:
\begin{align*}
    \epsilon_K^2 &= \sum_{\kappa > K}\sum_{\norm{k}_\infty = \kappa} \frac{ c_k^2 }{ \left[ 1 + \left(2\pi \norm{k}_2 \right)^2\right]^s } \\
    &\leq  \mathcal{C}(D)^2 \sum_{\kappa > K}^\infty \frac{ \sigma_\kappa - \sigma_{\kappa - 1}}{ \left[ 1 + \left(2\pi \kappa \right)\right]^s },
\end{align*}
where $\sigma_\kappa \triangleq \#\{ k \in \Z^D : \norm{k}_\infty \leq \kappa\}$ count the number of wavevectors in each $\infty$-ball. To bound $\sigma_\kappa$, we expand the polynomials in $\kappa$:
\begin{align*}
  \sigma_\kappa  - \sigma_{\kappa-1}   & = (2\kappa+1)^D - (2\kappa-1)^D \\
  &=  \sum_{d=0}^{D-1} \binom{D}{d} (2\kappa)^d [1 - (-1)^{D-d}] \\
  &\leq  D \binom{D}{\floor{D/2}} (2\kappa)^{D-1},
\end{align*}
as there is less than $D/2$ nonzero coefficients, all bounded by the largest coefficient $2\binom{D}{\floor{D/2}}$.

Choosing order  $s = (D+1)/2$, the error is bounded above by tails of sums of inverse squares
\begin{align*}
  \epsilon_K^2 &\leq \mathcal{C}(D)^2 D \binom{D}{\floor{D/2}} \sum_{\kappa > K}^\infty \frac{(2\kappa)^{D-1}}{
    \left[ 2\pi \kappa \right]^{D+1} } \\
  &\leq D \binom{D}{\floor{D/2}} \frac{\mathcal{C}(D)^2}{4\pi^{D+1}} \sum_{\kappa > K}^\infty \kappa^{-2}.
\end{align*}
As $\sum_{\kappa>K}^\infty \kappa^{-2} < \int_K^\infty u^{-2} = K^{-1}$, it follows that $\epsilon_K \in \mathcal O(K^{-1/2})$.

To bound the decay constant $\mathcal{G(D)}$, we use an upper bound for the  binomial coefficient, computed in \cite{Stanica:2001wy}:
\begin{align*}
  \binom{m n}{p n} < \frac{1}{\sqrt{2\pi n }} 
  \frac{m^{mn+1/2}}{(m-p)^{(m-p)n+1/2}p^{pn + 1/2}},
\end{align*}
for $m,n,p\in \N$ and $m>p$. Setting $n = \floor{D/2}+1$, $m=2$, $p=1$, provides a conservative estimate:
\begin{align*}
  \binom{D}{\floor{D/2}} < \sqrt{\frac{8}{\pi}}\frac{2^D}{\sqrt{D+1}}.
\end{align*}
As observables are $L_2$-normalized Fourier functions  \eqref{eq:fourierfunction}, $\mathcal{C}(D) = (2\pi)^{-D/2}$, allowing us to compute a bound on $\mathcal E^2(D)$:
\begin{align*}
  \mathcal E^2(D) &= D \binom{D}{\floor{D/2}} \frac{\mathcal{C}(D)^2}{4\pi^{D+1}} 
  < \sqrt{\frac{8}{\pi^3}} \frac{\sqrt{D}}{\pi^{2D}},
\shortintertext{resulting in}
\mathcal E(D) &< \left(\frac{2}{\pi}\right)^{3/4} \frac{ D^{1/4} }{\pi^{D}}.
\end{align*}
\end{proof}

\section{Diffusion Maps}
\label{sec:diffusionmaps}

The Diffusion Maps is an algorithm that retrieves an orthogonal coordinate system, the diffusion coordinates, for a set of points in a metric space. If the set of points is taken from a differentiable submanifold in the original space, the retrieved coordinates are discretized eigenfunctions of a Laplace-Beltrami operator on the submanifold. Diffusion coordinates yield a parsimonious description of the data set, i.e., even a truncated diffusion coordinate set provides a good low\hyp{}dimensional model of the original data set.

The Diffusion Maps is a part of a broader class of laplacian manifold\hyp{}learning algorithms, which rely on spectral computations to obtain information about geometry of manifolds. The algorithm was first described by Lafon and Coifman in \cite{Coifman:2005bk,Coifman:2006cy} with more detailed discussions on numerical implementation given in \cite{Lafon:2006bo,Lee:2010cz}.
To convey the intuition about the algorithm, we provide a short account of the general algorithm and implementation details we found relevant. At the end, we discuss how we applied it to the ergodic quotient.

The ambient space is the space $\C^N$ with a distance function $\mathcal D$. We aim to obtain a representation of a subset $\mathcal Y \subset \C^N$, with a distribution $dP(y) = p(y)dy$  supported on $\mathcal Y$. Averaging operators $A_h$ are defined through the kernel $k_h$:
\begin{align}
  A_h \phi(y) &= \frac{\int_{\C^N} k_h(y,z) \phi(z) dP(z)}{\int_{\C^N} k_h(y,z) dP(z)}\\
 k_h(y,z) &= (4\pi)^{-d/2} \exp\frac{-\mathcal{D}^2(y,z)}{4h}
\end{align}
and they represent diffusion of points on $\mathcal Y$, while the diffusion time-scale $h$ specifies the width of the gaussian kernel.\footnote{Sometimes $1/h$ is referred to as the \emph{bandwidth}.} The $h \to 0$ limit of the flow $A_h^{t/h}$ has a simple structure  when $\mathcal Y$ has a well-defined differential structure: it is generated by a Fokker-Planck-type operator $G$:
\begin{align*}
  \lim_{h \to 0} A_h^{t/h} &= e^{-t G}\\
  G &= \Delta - \frac{\nabla p}{p}, \notag
\end{align*}
where diffusion evolves over time period $t$. The Diffusion Maps algorithm computes values of eigenfunctions $\chi^{(k)}(y)$ of $e^{-tG}$ at a finite number of points $y_i$ that are sampled from $\mathcal Y$ according to the (sampling) density $p(y)$.

Eigenfunctions of the Laplace-Beltrami operator convey the geometry of the set $\mathcal Y$. If understanding geometry of $\mathcal Y$ is the goal, as in the case of ergodic quotient, it is necessary to remove the bias in $G$ resulting from non-uniformity in density $p$. The bias can be removed by replacing the heat kernel $k_h$ by its rescaled version:
\begin{align*}
\hat k_h &= \frac{ k_h(y,z) }{ \hat p(y) \hat p(z)} \\
\hat p(y) &= \int k_h(y,z)dz, \notag
\end{align*}
where $\hat p$ acts as an estimate for the sampling density.
In this case, the flow in the $h\to 0$ limit represents the pure Laplace-Beltrami diffusion on $\mathcal Y$, i.e., $G = \Delta$.

To compute the diffusion coordinates, defined as the embedding of points $y_i$ into eigenvalue-scaled eigenfunctions of the diffusion operator  $y \to (\lambda_1 \chi^{(1)}(y), \lambda_2 \chi^{(2)}(y), \dots)$, Diffusion Maps treats the finite set of sampling points $y_i \in \C^N$ as vertices of a fully-connected graph in $\mathcal Y$, while weights of edges are given by the ambient distance $\mathcal D$ on $\C^D$, between pairs of points $y_i$. The diffusion process on the graph can be represented by a Markov chain, whose eigenvectors $\chi^{(k)}$ are discretizations of eigenfunctions of the diffusion on $\mathcal Y$. In other words, we construct a discrete embedding $y_i \to (\lambda_1 \chi^{(1)}_i, \lambda_2 \chi^{(2)}_i, \dots, \lambda_N \chi^{(N)}_i)$, assuming $\norm{\chi^{(k)}}_\infty = 1$.

From the standpoint of the implementation, Diffusion Maps takes a matrix of pairwise distances between points $y_i$ as the input, and computes diffusion coordinates $\lambda \chi$ following these steps:
\begin{subequations}
  \begin{align}
    \text{$\mathcal D$ matrix}& & (d^2_{ij}) &= \mathcal D^2(y_i,y_j),\\
    \text{Heat kernel}& & A_{ij} &= \exp \left[
      -\frac{d^2_{ij}}{4h} \right],\label{eq:heatkernel}\\
    \text{Est. sampling density}& & \hat p_i &= \sum_{j=1}^N A_{ij},\\
    \text{Unbiased heat kernel}& & \hat A_{ij} &=
    \frac{A_{ij}}{p_i p_j},\\
    \text{State transition}& & S_{ij} &= \frac{\hat
      A_{ij}}{\sum_{j=1}^N \hat A_{ij}},\\
    \text{Solve the eigenproblem}& & S \chi &= \lambda
    \chi, \label{eq:eprob}
  \end{align}
\end{subequations}
with $i,j \in 1,2,\dots,N$ indexing elements of matrices and vectors, and $\mathcal{D}_{-s,K}$ representing summation \eqref{eq:negsobdist} truncated to $k\in [-K,K]^D$.

The time-scale of the heat kernel $h$ in \eqref{eq:heatkernel} is a parameter that determines the strength of diffusion along graph that samples $\mathcal Y$. It influences the information about topology of $\mathcal Y$ that we infer from discrete data:  too small of a choice of $h$ will introduce artificial discontinuities, while too large of a choice will artificially connect disconnected components of $\mathcal Y$. Paper \cite[][\S 5.3]{Lee:2010cz} describes several heuristics for determining $h$ from the data set; we chose the \emph{Neighborhood Size Stability} approach. It computes the minimal time-scale $h$ for which every graph vertex has $N_{min}$ neighbors within the characteristic diffusion distance  $\sqrt{2 h}$, as measured by the $H^{-s}$ distance. Number $N_{min}$ is chosen by the user, depending on number of discrete samples analyzed.

Since $S$ is not symmetric, complex values appear as partial results in numerical solutions of its eigenproblem, which is a source of further numerical errors. To avoid such issues, it is common to first symmetrize the matrix, using a spectrum-preserving symmetrization:
\begin{equation*}
\hat S_{ij} = \frac{\hat A_{ij}}
{\sqrt{\sum_{k=1}^N \hat A_{ik}} \sqrt{\sum_{k=1}^N \hat A_{kj}} }.
\end{equation*}
Solving the eigenproblem for $\hat S$, we obtain eigenvectors $\hat \chi^{(k)}$ from which we recover the eigenvectors that sample the diffusion eigenfunctions \eqref{eq:eprob} by rescaling $\hat \chi^{(k)}$ by the zeroth eigenvector $\chi^{(k)}_i = \hat \chi^{(k)}_i / \hat \chi^{(0)}_i$. Vectors $\chi^{(k)}$ are point-wise evaluations of diffusion eigenfunctions $\chi^{(k)}(y)$ at $y_i$, i.e., $\chi^{(k)}_i = \chi^{(k)}(y_i)$.

In application of Diffusion Maps to the ergodic quotient, initial conditions $x_i \in \mathcal M$ and associated trajectories get mapped to the points $y_i$ in the space $\C^N$ of averaged observables, with dimension $N$ given by the number of chosen observables. The ambient distance $\mathcal D$ is the $H^{-s}$ distance, while the set $\mathcal Y$ is the ergodic quotient $\xi$. The sampling density $p(y)$ is determined both by the distribution of initial conditions $x_i$ of trajectories on the state space, which can be controlled, and the distribution of the ergodic sets in the state space, which is a priori unknown and cannot be controlled. As a consequence, $p(y)$ is rarely uniform and the rescaling of the diffusion kernel is a necessary step in the algorithm to obtain the intrinsic coordinate set with the correct interpretation of the euclidean distance. 

\bibliography{manuscript}

\begin{thebibliography}{41}
\expandafter\ifx\csname natexlab\endcsname\relax\def\natexlab#1{#1}\fi
\providecommand{\bibinfo}[2]{#2}
\ifx\xfnm\relax \def\xfnm[#1]{\unskip,\space#1}\fi
\bibitem[{Alted et~al.(2002)Alted, Vilata and {others}}]{alted02:pytables}
\bibinfo{author}{F.~Alted}, \bibinfo{author}{I.~Vilata},
  \bibinfo{author}{{others}}, \bibinfo{title}{{PyTables: Hierarchical Datasets
  in Python}}, \bibinfo{type}{Technical Report}, \bibinfo{year}{2002}.
\bibitem[{Bolsinov and Fomenko(2004)}]{Fomenko:2004}
\bibinfo{author}{A.V. Bolsinov}, \bibinfo{author}{A.T. Fomenko},
  \bibinfo{title}{{Integrable Hamiltonian systems: geometry, topology,
  classification}}, \bibinfo{publisher}{CRC Press}, \bibinfo{year}{2004}.
\bibitem[{Budi{\v s}i{\'c} and Mezi{\'c}(2009)}]{Budisic:2009iy}
\bibinfo{author}{M.~Budi{\v s}i{\'c}}, \bibinfo{author}{I.~Mezi{\'c}},
  \bibinfo{title}{{An approximate parametrization of the ergodic partition
  using time averaged observables}}, \bibinfo{journal}{Proceedings of the 48th
  IEEE Conference on Decision and Control, held jointly with the 2009 28th
  Chinese Control Conference. CDC/CCC 2009.}  (\bibinfo{year}{2009})
  \bibinfo{pages}{3162--3168}.
\bibitem[{Budi{\v s}i{\'c} and Mezi{\'c}(2011)}]{Budisic:2011ky}
\bibinfo{author}{M.~Budi{\v s}i{\'c}}, \bibinfo{author}{I.~Mezi{\'c}},
  \bibinfo{title}{{A Bifurcation in an Unsteady 3D Fluid Flow Found Using the
  Ergodic Quotient}}, in: \bibinfo{booktitle}{IUTAM Symposium on 50 Years of
  Chaos: Applied and Theoretical (poster presentation)},
  \bibinfo{address}{Kyoto}.
\bibitem[{Clewley et~al.(2004)Clewley, Sherwood, LaMar and
  Guckenheimer}]{clewley04:pydstool}
\bibinfo{author}{R.H. Clewley}, \bibinfo{author}{W.E. Sherwood},
  \bibinfo{author}{M.D. LaMar}, \bibinfo{author}{J.G. Guckenheimer},
  \bibinfo{title}{{PyDSTool dynamical systems software}},
  \bibinfo{type}{Technical Report}, \bibinfo{year}{2004}.
\bibitem[{Coifman and Lafon(2006)}]{Coifman:2006cy}
\bibinfo{author}{R.~Coifman}, \bibinfo{author}{S.~Lafon},
  \bibinfo{title}{{Diffusion maps}}, \bibinfo{journal}{Applied and
  Computational Harmonic Analysis} \bibinfo{volume}{21} (\bibinfo{year}{2006})
  \bibinfo{pages}{5--30}.
\bibitem[{Coifman et~al.(2005)Coifman, Lafon, Lee, Maggioni, Nadler, Warner and
  Zucker}]{Coifman:2005bk}
\bibinfo{author}{R.~Coifman}, \bibinfo{author}{S.~Lafon}, \bibinfo{author}{A.B.
  Lee}, \bibinfo{author}{M.~Maggioni}, \bibinfo{author}{B.~Nadler},
  \bibinfo{author}{F.~Warner}, \bibinfo{author}{S.W. Zucker},
  \bibinfo{title}{{Geometric diffusions as a tool for harmonic analysis and
  structure definition of data: Diffusion maps }},
  \bibinfo{journal}{Proceedings of National Academy of Sciences}
  \bibinfo{volume}{102} (\bibinfo{year}{2005}) \bibinfo{pages}{7426--7431}.
\bibitem[{Cvitanovi{\'c} et~al.(2010)Cvitanovi{\'c}, Artuso, Mainieri, Tanner
  and Vattay}]{ChaosBook}
\bibinfo{author}{P.~Cvitanovi{\'c}}, \bibinfo{author}{R.~Artuso},
  \bibinfo{author}{R.~Mainieri}, \bibinfo{author}{G.~Tanner},
  \bibinfo{author}{G.~Vattay}, \bibinfo{title}{{Chaos: Classical and Quantum}},
  \bibinfo{publisher}{Niels Bohr Institute}, \bibinfo{address}{Copenhagen},
  \bibinfo{edition}{stable 13th} edition, \bibinfo{year}{2010}.
\bibitem[{Dombre et~al.(1986)Dombre, Frisch, Greene, H{\'e}non, Mehr and
  Soward}]{Dombre:1986td}
\bibinfo{author}{T.~Dombre}, \bibinfo{author}{U.~Frisch}, \bibinfo{author}{J.M.
  Greene}, \bibinfo{author}{M.~H{\'e}non}, \bibinfo{author}{A.~Mehr},
  \bibinfo{author}{A.M. Soward}, \bibinfo{title}{{Chaotic streamlines in the
  ABC flows}}, \bibinfo{journal}{Journal of Fluid Mechanics}
  \bibinfo{volume}{167} (\bibinfo{year}{1986}) \bibinfo{pages}{353--391}.
\bibitem[{Froyland and Dellnitz(2003)}]{Froyland:2003jj}
\bibinfo{author}{G.~Froyland}, \bibinfo{author}{M.~Dellnitz},
  \bibinfo{title}{{Detecting and locating near-optimal almost-invariant sets
  and cycles}}, \bibinfo{journal}{SIAM Journal on Scientific Computing}
  \bibinfo{volume}{24} (\bibinfo{year}{2003}) \bibinfo{pages}{1839--1863
  (electronic)}.
\bibitem[{Froyland et~al.(2011)Froyland, Junge and Koltai}]{Froyland:2011tt}
\bibinfo{author}{G.~Froyland}, \bibinfo{author}{O.~Junge},
  \bibinfo{author}{P.~Koltai}, \bibinfo{title}{{Estimating long term behavior
  of flows without trajectory integration: the infinitesimal generator
  approach}}, \bibinfo{journal}{arXiv.org} \bibinfo{volume}{math.NA}
  (\bibinfo{year}{2011}).
\bibitem[{Froyland and Padberg(2009)}]{Froyland:2009ti}
\bibinfo{author}{G.~Froyland}, \bibinfo{author}{K.~Padberg},
  \bibinfo{title}{{Almost-invariant sets and invariant manifolds--connecting
  probabilistic and geometric descriptions of coherent structures in flows}},
  \bibinfo{journal}{Physica D. Nonlinear Phenomena} \bibinfo{volume}{238}
  (\bibinfo{year}{2009}) \bibinfo{pages}{1507--1523}.
\bibitem[{Froyland et~al.(2010)Froyland, Santitissadeekorn and
  Monahan}]{Froyland:2010jo}
\bibinfo{author}{G.~Froyland}, \bibinfo{author}{N.~Santitissadeekorn},
  \bibinfo{author}{A.~Monahan}, \bibinfo{title}{{Transport in time-dependent
  dynamical systems: Finite-time coherent sets}}, \bibinfo{journal}{Chaos: An
  Interdisciplinary Journal of Nonlinear Science} \bibinfo{volume}{20}
  (\bibinfo{year}{2010}) \bibinfo{pages}{043116}.
\bibitem[{Haller(2001)}]{Haller:2001vh}
\bibinfo{author}{G.~Haller}, \bibinfo{title}{{Distinguished material surfaces
  and coherent structures in three-dimensional fluid flows}},
  \bibinfo{journal}{Physica D. Nonlinear Phenomena} \bibinfo{volume}{149}
  (\bibinfo{year}{2001}) \bibinfo{pages}{248--277}.
\bibitem[{Haller(2002)}]{Haller:2002bf}
\bibinfo{author}{G.~Haller}, \bibinfo{title}{{Lagrangian coherent structures
  from approximate velocity data}}, \bibinfo{journal}{Physics of Fluids}
  \bibinfo{volume}{14} (\bibinfo{year}{2002}) \bibinfo{pages}{1851--1861}.
\bibitem[{Haller(2011)}]{Haller:2011kr}
\bibinfo{author}{G.~Haller}, \bibinfo{title}{{A variational theory of
  hyperbolic Lagrangian Coherent Structures}}, \bibinfo{journal}{Physica D.
  Nonlinear Phenomena} \bibinfo{volume}{240} (\bibinfo{year}{2011})
  \bibinfo{pages}{574--598}.
\bibitem[{Haller and Yuan(2000)}]{Haller:2000us}
\bibinfo{author}{G.~Haller}, \bibinfo{author}{G.~Yuan},
  \bibinfo{title}{{Lagrangian coherent structures and mixing in two-dimensional
  turbulence}}, \bibinfo{journal}{Physica D. Nonlinear Phenomena}
  \bibinfo{volume}{147} (\bibinfo{year}{2000}) \bibinfo{pages}{352--370}.
\bibitem[{Jakobson et~al.(2001)Jakobson, Nadirashvili and
  Toth}]{Jakobson:2001cf}
\bibinfo{author}{D.~Jakobson}, \bibinfo{author}{N.~Nadirashvili},
  \bibinfo{author}{J.~Toth}, \bibinfo{title}{{Geometric properties of
  eigenfunctions}}, \bibinfo{journal}{Russian Mathematical Surveys}
  \bibinfo{volume}{56} (\bibinfo{year}{2001}) \bibinfo{pages}{1085}.
\bibitem[{Jones et~al.(2001)Jones, Oliphant, Peterson and
  {others}}]{jones01:scipy}
\bibinfo{author}{E.~Jones}, \bibinfo{author}{T.~Oliphant},
  \bibinfo{author}{P.~Peterson}, \bibinfo{author}{{others}},
  \bibinfo{title}{{SciPy: Open source scientific tools for Python}},
  \bibinfo{type}{Technical Report}, \bibinfo{year}{2001}.
\bibitem[{Jones et~al.(2008)Jones, Maggioni and Schul}]{Jones:2008cy}
\bibinfo{author}{P.W. Jones}, \bibinfo{author}{M.~Maggioni},
  \bibinfo{author}{R.~Schul}, \bibinfo{title}{{Manifold parametrizations by
  eigenfunctions of the Laplacian and heat kernels}},
  \bibinfo{journal}{Proceedings of National Academy of Sciences}
  \bibinfo{volume}{105} (\bibinfo{year}{2008}) \bibinfo{pages}{1803--1808}.
\bibitem[{Joseph and Legras(2002)}]{Joseph:2002vi}
\bibinfo{author}{B.~Joseph}, \bibinfo{author}{B.~Legras},
  \bibinfo{title}{{Relation between kinematic boundaries, stirring, and
  barriers for the Antarctic polar vortex}}, \bibinfo{journal}{Journal Of The
  Atmospheric Sciences} \bibinfo{volume}{59} (\bibinfo{year}{2002})
  \bibinfo{pages}{1198--1212}.
\bibitem[{Katok and Hasselblatt(1995)}]{Katok:1995ul}
\bibinfo{author}{A.~Katok}, \bibinfo{author}{B.~Hasselblatt},
  \bibinfo{title}{{Introduction to the Modern Theory of Dynamical Systems}},
  \bibinfo{publisher}{Cambridge University Press}, \bibinfo{address}{New York,
  NY, USA}, \bibinfo{year}{1995}.
\bibitem[{Lafon and Lee(2006)}]{Lafon:2006bo}
\bibinfo{author}{S.~Lafon}, \bibinfo{author}{A.B. Lee},
  \bibinfo{title}{{Diffusion maps and coarse-graining: a unified framework for
  dimensionality reduction, graph partitioning, and data set
  parameterization}}, \bibinfo{journal}{Pattern Analysis and Machine
  Intelligence, IEEE Transactions on} \bibinfo{volume}{28}
  (\bibinfo{year}{2006}) \bibinfo{pages}{1393--1403}.
\bibitem[{Lee and Wasserman(2010)}]{Lee:2010cz}
\bibinfo{author}{A.B. Lee}, \bibinfo{author}{L.~Wasserman},
  \bibinfo{title}{{Spectral connectivity analysis}}, \bibinfo{journal}{Journal
  of the American Statistical Association} \bibinfo{volume}{105}
  (\bibinfo{year}{2010}) \bibinfo{pages}{1241--1255}.
\bibitem[{Levnaji{\'c} and Mezi{\'c}(2010)}]{Levnajic:2010gq}
\bibinfo{author}{Z.~Levnaji{\'c}}, \bibinfo{author}{I.~Mezi{\'c}},
  \bibinfo{title}{{Ergodic theory and visualization. I. Mesochronic plots for
  visualization of ergodic partition and invariant sets}},
  \bibinfo{journal}{Chaos: An Interdisciplinary Journal of Nonlinear Science}
  \bibinfo{volume}{20} (\bibinfo{year}{2010}) \bibinfo{pages}{19}.
\bibitem[{Mathew and Mezi{\'c}(2011)}]{Mathew:2011ev}
\bibinfo{author}{G.~Mathew}, \bibinfo{author}{I.~Mezi{\'c}},
  \bibinfo{title}{{Metrics for ergodicity and design of ergodic dynamics for
  multi-agent systems}}, \bibinfo{journal}{Physica D. Nonlinear Phenomena}
  \bibinfo{volume}{240} (\bibinfo{year}{2011}) \bibinfo{pages}{432--442}.
\bibitem[{Mezi{\'c} and Banaszuk(2004)}]{Mezic:2004is}
\bibinfo{author}{I.~Mezi{\'c}}, \bibinfo{author}{A.~Banaszuk},
  \bibinfo{title}{{Comparison of systems with complex behavior}},
  \bibinfo{journal}{Physica D. Nonlinear Phenomena} \bibinfo{volume}{197}
  (\bibinfo{year}{2004}) \bibinfo{pages}{101--133}.
\bibitem[{Mezi{\'c} et~al.(2010)Mezi{\'c}, Loire, Fonoberov and
  Hogan}]{Mezic:2010kh}
\bibinfo{author}{I.~Mezi{\'c}}, \bibinfo{author}{S.~Loire},
  \bibinfo{author}{V.A. Fonoberov}, \bibinfo{author}{P.J. Hogan},
  \bibinfo{title}{{A New Mixing Diagnostic and Gulf Oil Spill Movement}},
  \bibinfo{journal}{Science Magazine} \bibinfo{volume}{330}
  (\bibinfo{year}{2010}) \bibinfo{pages}{science.1194607v2}.
\bibitem[{Mezi{\'c} and Wiggins(1994)}]{Mezic:1994ez}
\bibinfo{author}{I.~Mezi{\'c}}, \bibinfo{author}{S.~Wiggins},
  \bibinfo{title}{{On the integrability and perturbation of 3-dimensional
  fluid-flows with symmetry}}, \bibinfo{journal}{Journal of Nonlinear Science}
  \bibinfo{volume}{4} (\bibinfo{year}{1994}) \bibinfo{pages}{157--194}.
\bibitem[{Mezi{\'c} and Wiggins(1999)}]{Mezic:1999fu}
\bibinfo{author}{I.~Mezi{\'c}}, \bibinfo{author}{S.~Wiggins},
  \bibinfo{title}{{A method for visualization of invariant sets of dynamical
  systems based on the ergodic partition}}, \bibinfo{journal}{Chaos: An
  Interdisciplinary Journal of Nonlinear Science} \bibinfo{volume}{9}
  (\bibinfo{year}{1999}) \bibinfo{pages}{213--218}.
\bibitem[{Perry and Wiggins(1994)}]{Perry:1994wl}
\bibinfo{author}{A.D. Perry}, \bibinfo{author}{S.~Wiggins},
  \bibinfo{title}{{KAM tori are very sticky: rigorous lower bounds on the time
  to move away from an invariant Lagrangian torus with linear flow}},
  \bibinfo{journal}{Physica D. Nonlinear Phenomena} \bibinfo{volume}{71}
  (\bibinfo{year}{1994}) \bibinfo{pages}{102--121}.
\bibitem[{Petersen(1989)}]{Petersen:1989ly}
\bibinfo{author}{K.~Petersen}, \bibinfo{title}{{Ergodic theory}},
  \bibinfo{publisher}{Cambridge University Press}, \bibinfo{address}{Cambridge,
  UK}, \bibinfo{year}{1989}.
\bibitem[{Shadden et~al.(2005)Shadden, Lekien and Marsden}]{Shadden:2005vn}
\bibinfo{author}{S.C. Shadden}, \bibinfo{author}{F.~Lekien},
  \bibinfo{author}{J.E. Marsden}, \bibinfo{title}{{Definition and properties of
  Lagrangian coherent structures from finite-time Lyapunov exponents in
  two-dimensional aperiodic flows}}, \bibinfo{journal}{Physica D. Nonlinear
  Phenomena} \bibinfo{volume}{212} (\bibinfo{year}{2005})
  \bibinfo{pages}{271--304}.
\bibitem[{Sigurgeirsson and Stuart(2001)}]{Sigurgeirsson:2001wi}
\bibinfo{author}{H.~Sigurgeirsson}, \bibinfo{author}{A.M. Stuart},
  \bibinfo{title}{{Statistics From Computations}},
  \bibinfo{journal}{Foundations of Computational Mathematics, LMS Lecture Note
  Series, Cambridge University Press} \bibinfo{volume}{284}
  (\bibinfo{year}{2001}) \bibinfo{pages}{323--344}.
\bibitem[{Stanica(2001)}]{Stanica:2001wy}
\bibinfo{author}{P.~Stanica}, \bibinfo{title}{{Good lower and upper bounds on
  binomial coefficients}}, \bibinfo{journal}{Journal of Inequalities in Pure
  and Applied Mathematics} \bibinfo{volume}{2} (\bibinfo{year}{2001})
  \bibinfo{pages}{Article 30, 5 pp. (electronic)}.
\bibitem[{Susuki et~al.(2011)Susuki, Mezi{\'c} and Hikihara}]{Susuki:2011jz}
\bibinfo{author}{Y.~Susuki}, \bibinfo{author}{I.~Mezi{\'c}},
  \bibinfo{author}{T.~Hikihara}, \bibinfo{title}{{Coherent Swing Instability of
  Power Grids}}, \bibinfo{journal}{Journal of Nonlinear Science}
  (\bibinfo{year}{2011}).
\bibitem[{Treschev(1998)}]{Treschev:1998uq}
\bibinfo{author}{D.~Treschev}, \bibinfo{title}{{Width of stochastic layers in
  near-integrable two-dimensional symplectic maps}}, \bibinfo{journal}{Physica
  D. Nonlinear Phenomena} \bibinfo{volume}{116} (\bibinfo{year}{1998})
  \bibinfo{pages}{21--43}.
\bibitem[{Uhlenbeck(1976)}]{Uhlenbeck:1976wb}
\bibinfo{author}{K.~Uhlenbeck}, \bibinfo{title}{{Generic properties of
  eigenfunctions}}, \bibinfo{journal}{American Journal of Mathematics}
  \bibinfo{volume}{98} (\bibinfo{year}{1976}) \bibinfo{pages}{1059--1078}.
\bibitem[{Vaidya and Mezi{\'c}(2011)}]{Vaidya:2011wo}
\bibinfo{author}{U.~Vaidya}, \bibinfo{author}{I.~Mezi{\'c}},
  \bibinfo{title}{{Existence of invariant tori in action-angle-angle maps with
  degeneracy}}, \bibinfo{year}{2011}.
\bibitem[{Von~Luxburg(2007)}]{Luxburg:2007bb}
\bibinfo{author}{U.~Von~Luxburg}, \bibinfo{title}{{A tutorial on spectral
  clustering}}, \bibinfo{journal}{Statistics and Computing}
  \bibinfo{volume}{17} (\bibinfo{year}{2007}) \bibinfo{pages}{395--416}.
\bibitem[{Wiener and Wintner(1941)}]{Wiener:1941wy}
\bibinfo{author}{N.~Wiener}, \bibinfo{author}{A.~Wintner},
  \bibinfo{title}{{Harmonic Analysis and Ergodic Theory}},
  \bibinfo{journal}{American Journal of Mathematics} \bibinfo{volume}{63}
  (\bibinfo{year}{1941}) \bibinfo{pages}{415--426}.

\end{thebibliography}
\bibliographystyle{model1b-num-names}

\end{document}